\documentclass[oneside,english]{amsart}
\usepackage[T1]{fontenc}
\usepackage[latin9]{inputenc}
\usepackage{geometry}
\usepackage{color}
\usepackage{babel}
\usepackage{verbatim}
\usepackage{mathrsfs}
\usepackage{enumitem}
\usepackage{amstext}
\usepackage{amsmath}
\usepackage{amsthm}
\usepackage{amssymb}
\usepackage[all]{xy}
\usepackage[unicode=true,pdfusetitle,
 bookmarks=true,bookmarksnumbered=false,bookmarksopen=false,
 breaklinks=false,pdfborder={0 0 1},backref=false,colorlinks=true]
 {hyperref}

\makeatletter

\newcommand{\DD}{\mathbb{D}}
\newcommand{\PP}{\mathbb{P}}

\newcommand{\ZZ}{{\mathbb Z}}

\newcommand{\QQ}{{\mathbb Q}}

\newcommand{\CC}{{\mathbb C}}

\renewcommand{\AA}{{\mathbb A}}

\newcommand{\calO}{{\mathcal O}}
\newcommand{\calE}{{\mathcal E}}
\newcommand{\calL}{{\mathcal L}}

\newcommand{\calM}{{\mathcal M}}

\newcommand{\calF}{{\mathcal F}}
\newcommand{\calG}{{\mathcal G}}

\newcommand{\calH}{{\mathcal H}}
\newcommand{\calC}{{\mathcal C}}

\newcommand{\calA}{{\mathcal A}}

\newcommand{\scrM}{\mathscr{M}}

\DeclareMathOperator{\Spec}{Spec}

\newcommand{\End}{\textrm{End}}

\newcommand{\FF}{\mathbb{F}}
\newcommand{\Fq}{\mathbb{F}_q}
\newcommand{\Fp}{\mathbb{F}_p}

\newcommand{\Aut}{\textrm{Aut}}
\newcommand{\Hom}{\textrm{Hom}}

\newcommand{\Zp}{\mathbb{Z}_p}

\newcommand{\Qp}{\mathbb{Q}_p}
\newcommand{\Qpbar}{\overline{\mathbb Q}_p}
\newcommand{\Ql}{\mathbb{Q}_l}
\newcommand{\Qlbar}{\overline{\mathbb{Q}}_l}
\newcommand{\Qlambar}{\overline{\mathbb Q}_{\lambda}}
\newcommand{\Qlambarp}{\overline{\mathbb Q}_{\lambda'}}
\newcommand{\fisoc}[1]{\textbf{F-Isoc}(#1)}
\newcommand{\fisocd}[1]{\textbf{F-Isoc}^{\dagger}(#1)}
\newcommand{\weil}[1]{\textbf{Weil}(#1)}

\numberwithin{equation}{section}
\numberwithin{figure}{section}

\theoremstyle{plain}
\newtheorem{thm}{\protect\theoremname}[section]
\theoremstyle{empty}

\theoremstyle{plain}
\newtheorem{conjecture}[thm]{\protect\conjecturename}
\theoremstyle{remark}
\newtheorem{rem}[thm]{\protect\remarkname}
\theoremstyle{plain}
\newtheorem*{thm*}{\protect\theoremname}
\theoremstyle{plain}
\newtheorem*{cor*}{\protect\corollaryname}
\theoremstyle{remark}
\newtheorem*{acknowledgement*}{\protect\acknowledgementname}
\theoremstyle{definition}
\newtheorem{defn}[thm]{\protect\definitionname}
\theoremstyle{plain}
\newtheorem{fact}[thm]{\protect\factname}
\theoremstyle{remark}
\newtheorem{notation}[thm]{\protect\notationname}
\theoremstyle{definition}
\newtheorem{example}[thm]{\protect\examplename}
\theoremstyle{plain}
\newtheorem{cor}[thm]{\protect\corollaryname}
\theoremstyle{remark}
\newtheorem*{rem*}{\protect\remarkname}
\theoremstyle{plain}
\newtheorem{lem}[thm]{\protect\lemmaname}
\newlist{casenv}{enumerate}{4}
\setlist[casenv]{leftmargin=*,align=left,widest={iiii}}
\setlist[casenv,1]{label={{\itshape\ \casename} \arabic*.},ref=\arabic*}
\setlist[casenv,2]{label={{\itshape\ \casename} \roman*.},ref=\roman*}
\setlist[casenv,3]{label={{\itshape\ \casename\ \alph*.}},ref=\alph*}
\setlist[casenv,4]{label={{\itshape\ \casename} \arabic*.},ref=\arabic*}

\theoremstyle{plain}
\newtheorem{prop}[thm]{\protect\propositionname}
\theoremstyle{plain}

\usepackage[all]{xy}
\usepackage{url}
\theoremstyle{definition}

\numberwithin{equation}{thm}

\makeatother

\providecommand{\acknowledgementname}{Acknowledgement}
\providecommand{\casename}{Case}
\providecommand{\conjecturename}{Conjecture}
\providecommand{\corollaryname}{Corollary}
\providecommand{\definitionname}{Definition}
\providecommand{\examplename}{Example}
\providecommand{\factname}{Fact}
\providecommand{\lemmaname}{Lemma}
\providecommand{\notationname}{Notation}
\providecommand{\propositionname}{Proposition}
\providecommand{\questionname}{Question}
\providecommand{\remarkname}{Remark}
\providecommand{\theoremname}{Theorem}

\begin{document}

\title{Rank 2 Local Systems and Abelian Varieties}

\author{Raju Krishnamoorthy}
\email{raju@uga.edu }  
\address{Department of Mathematics, University of Georgia, Athens, GA 30605, USA}
\author{Ambrus P\'al}
\email{a.pal@imperial.ac.uk}
\address{Department of Mathematics,
180 Queens Gate, Imperial College, London, SW7 2AZ,
United Kingdom}
\begin{abstract}
Let $X/\mathbb{F}_{q}$ be a smooth, geometrically connected variety. For $X$ projective, we prove a Lefschetz-style theorem for abelian schemes of $\text{GL}_2$-type on $X$, modeled after a theorem of Simpson. Inspired by work of Corlette-Simpson over $\mathbb{C}$, we formulate a conjecture that absolutely irreducible rank 2 local systems with infinite monodromy on $X$ come from families of abelian varieties.  We have the following application of our main result. If one assumes a strong form of Deligne's ($p$-adic) \emph{companions conjecture} from Weil II, then our conjecture for projective varieties reduces to the conjecture for projective curves. We also answer affirmitavely a question of Grothendieck on extending abelian schemes via their $p$-divisible groups.
\end{abstract}

\maketitle

\tableofcontents{}

\section{Introduction}
The goal of this article is to prove the following Lefschetz-style theorem. If $X/k$ is a smooth variety over a perfect field, $\fisoc{X}_{\Qpbar}$ denotes the category of $F$-isocrystals on $X$ with coefficients in $\Qpbar$. On a smooth \emph{proper} variety $X/\Fq$, an $F$-isocrystal is a $p$-adic analog of a lisse $l$-adic sheaf.
\begin{thm*}[\ref{Theorem:main_lefschetz}] Let $X/\Fq$ be a smooth projective variety. Then there exists an open subset $U\subset X$, whose complement has codimension at least 2, such that the following holds.

Let $C\subset U$ be a smooth projective curve that is the complete intersection of smooth ample divisors of $X$. Let $\pi_C\colon A_C\rightarrow C$ be an abelian scheme of $\text{GL}_2$-type: for a prime $l\neq p$, $R^1(\pi_C)_*\Qlbar$ has irreducible summands that have rank 2 and determinant $\Qlbar(-1)$. Then the following are equivalent.

\begin{itemize}
\item There exists an abelian scheme of $\text{GL}_2$-type $B_U\rightarrow U$ with $B_C\rightarrow C$ isogenous to $A_C\rightarrow C$.
\item The $F$-isocrystal $\mathbb{D}(A_C[p^{\infty}])\otimes \Qpbar\in \fisoc{C}_{\Qpbar}$ extends to an $F$-isocrystal $\calE\in \fisoc{X}_{\Qpbar}$.
\end{itemize}
\end{thm*}
Here, if $G\rightarrow S$ is a $p$-divisible group in characteristic $p$, then $\mathbb{D}(G)$ denotes the contravariant Dieudonn\'e crystal attached to $G$. This theorem is modeled on the following, which easily follows from a very special case of a corollary of Simpson.
\begin{thm}\label{Theorem:simpson_lefschetz}\cite[Corollary 4.3]{simpsonhiggs} Let $X/\CC$ be a smooth projective variety and $C\subset X$ a smooth curve that is the complete intersection of smooth ample divisors. Let $\pi_C\colon A_C\rightarrow C$ be an abelian scheme and set $L_C:=R^1(\pi^{an}_C)_*\CC$. Then the following are equivalent.\begin{itemize}
\item There exists an abelian scheme $\pi_X\colon A_X\rightarrow X$ extending $A_C\rightarrow C$.
\item The local system $L_C$ extends to a local system $L_X$ on $X$.
\end{itemize}
\end{thm}
Therefore, Theorem \ref{Theorem:main_lefschetz} is an analog of Theorem \ref{Theorem:simpson_lefschetz} over $\Fq$ for abelian schemes of $\text{GL}_2$-type. In contrast to Theorem \ref{Theorem:simpson_lefschetz}, Theorem \ref{Theorem:main_lefschetz} has an intervening $U\subset X$ and also potentially an isogeny.

The authors conjecture the following. For the definition of $l$-adic companions, see Remark \ref{Remark:companions_l-adic}.

\begin{conjecture}
\label{Conjecture:R2}(Conjecture R2) Let $X/\Fq$ be a smooth, geometrically connected, quasi-projective variety, let $l\neq p$ be a prime, and let $L$ be a lisse $\Qlbar$-sheaf of rank 2 such that
\begin{itemize}
\item $L$ has determinant $\Qlbar(-1)$ and
\item $L$ is irreducible with infinite geometric monodromy.
\end{itemize}
Then $L$ comes from a family of abelian varieties: there exists a non-empty open $U\subset X$ together with an abelian scheme
\[
\pi\colon A_{U}\rightarrow U
\]
 such that
\[
R^{1}\pi_{*}\Qlbar\cong{\displaystyle \bigoplus}(^{\sigma}L|_U)^{m}
\]
where $^{\sigma}L$ runs over the $l$-adic companions of $L$ and $m\in\mathbb{N}$.
\end{conjecture}

In particular, we conjecture such $L$ come from $H^1$ of an abelian scheme on an open $U\subset X$. The main evidence for Conjecture \ref{Conjecture:R2} comes from Drinfeld's first work on the Langlands correspondence.
\begin{thm}
\label{Theorem:GL2}(Drinfeld) Let $C/\Fq$ be a smooth affine curve and let $L$ be as in Conjecture \ref{Conjecture:R2}. Suppose $L$ has infinite (geometric) monodromy around some point at $\infty\in\overline{C}\backslash C$. Then $L$ comes from a family of abelian varieties in the following sense: let $E$ be the field generated by the Frobenius traces of $L$ and suppose $[E:\QQ]=g$. Then there exists an abelian scheme
\[
\pi\colon A_{C}\rightarrow C
\]
of dimension $g$ and an isomorphism $E\cong \textrm{End}_{C}(A)\otimes\QQ$, realizing $A_C$ as a $GL_{2}$-type abelian scheme, such that $L$ occurs as a summand of $R^1\pi_*\Qlbar$. Moreover, $A_{C}\rightarrow C$ is totally degenerate around $\infty$.
\end{thm}

See \cite[Proof of Proposition 19, Remark 20]{snowden2018constructing} for how to recover this result from Drinfeld's work. (This amounts to combining \cite[Main Theorem, Remark 5]{drinfeld1983} with \cite[Theorem 1]{drinfeld1977}.)
For general smooth $X/\Fq$, an \emph{overconvergent $F$-isocyrstal} is a good $p$-adic analog of a lisse $l$-adic sheaf. Conjecture \ref{Conjecture:R2} can then be formulated in the $l=p$ case, replacing $L$ with $\calE\in\fisocd{X}_{\Qpbar}$, an overconvergent $F$-isocrystal with coefficients in $\Qpbar$. If $X$ is projective, then any $F$-isocrystal is automatically overconvergent. Combined with a refined form of Deligne's $p$-adic companions conjecture, we obtain the following application to Conjecture \ref{Conjecture:R2}.
\begin{cor*}[\ref{Corollary:rank_2_reduces_to_curves}] Let $X/\Fq$ be a smooth projective variety with $\dim(X)\geq 2$ and let $L$ be a rank 2 lisse $\Qlbar$ sheaf with cyclotomic determinant and infinite geometric monodromy. Then there exists an open subset $U\subset X$, whose complement has codimension at least 2, such that
\begin{itemize}
\item if $C\subset U$ is a smooth proper curve that is the complete intersection of smooth ample divisors;
\item if $L_C$ comes from an abelian scheme on $A_C\rightarrow C$, in the sense of Conjecture \ref{Conjecture:R2}; and
\item if all $p$-adic companions to $L$ exist,
\end{itemize} then $L_U$ comes from an abelian scheme $B_U\rightarrow U$, i.e., Conjecture \ref{Conjecture:R2} is true for $(X,L)$
\end{cor*}
In other words, if one believes the $p$-adic companions conjecture, then Conjecture \ref{Conjecture:R2} for a projective variety reduces to the case of a single sufficiently generic curve $C\subset X$. See Definition \ref{Definition:complete_set_of_companions} for the definition of a \emph{complete set of $p$-adic companions}. In general, the existence of a complete set of $p$-adic companions is a strong form of Deligne's \emph{petits camarades cristallin} conjecture: see Conjectures \ref{Conjecture:Deligne} and \ref{Conjecture:Companions}. When we place certain conditions on the splitting of $p$ in $E$, the field of traces of $\mathcal{E}$, the existence of a single $p$-adic
companion guarantees the existence of all of them by Corollary \ref{Corollary:all_companions_Galois_twists}.

We briefly describe the strategy. First, we construct a (non-canonical) $p$-divisible group on $U\subset X$. Then, we use Serre-Tate theory to construct a formal (polarizable) abelian scheme over the formal scheme $X_{/C}$. Here the positivity of $C\subset X$ is used. Finally, using work of Grothendieck, Hartshorne, and Hironaka and the positivity of $C\subset X$, we globalize the family. As a key step, we record an affirmative answer to a question of Grothendieck \cite[4.9]{grothendieck1966theoreme}:
\begin{thm*}
(Corollary \ref{Corollary:grothendieck}) Let $X$ be a locally noetherian normal scheme and $U\subset X$ be an open dense subset whose complement has characteristic $p$. Let $A_{U}\rightarrow U$ be an abelian scheme. Then $A_{U}$ extends to an abelian scheme over $X$ if and only if $A_{U}[p^{\infty}]$ extends to a $p$-divisible group over $X$.
\end{thm*}
Combined with algebraization techniques, there is the following useful consequence, which is a $p$-adic analog of Simpson's Theorem \ref{Theorem:simpson_lefschetz}. 
\begin{cor*}
(Corollary \ref{Corollary:AV_extends_BT_extends}) Let $X/\Fq$ be a smooth projective variety and let $C\subset X$ be a smooth curve that is the complete intersection of smooth ample divisors. Let $A_{C}\rightarrow C$ be an abelian scheme. Suppose there exists a Zariski open neighborhood $U\supset C$ of $X$ such that $A_{C}[p^{\infty}]$ extends to a Barsotti-Tate group $\calG_{U}$ on $U$. Then there exists a unique abelian scheme $A_{U}\rightarrow U$, extending $A_{C}$, such that $A_{U}[p^{\infty}]\cong\mathcal{G}_{U}$. 
\end{cor*}

Remark \ref{Remark:extension_dimension_bounds} shows that hypothesis on the dimension is necessary. The proof of Corollary \ref{Corollary:AV_extends_BT_extends} makes use of $p$-to-$l$ companions. See Remark \ref{Remark:extend_AV_general_field} and Corollary \ref{Corollary:extend_AV_extend_BT_ii} for what we know over general fields $k$ of characteristic $p$. We now make some general remarks on Conjecture \ref{Conjecture:R2}.
\begin{rem}Our motivation to formulate Conjecture \ref{Conjecture:R2} partly comes from the following celebrated theorem of Corlette-Simpson \cite[Theorem 11.2]{corlette2008classification}: 
\begin{thm*}
\label{Theorem:CS}(Corlette-Simpson) Let $X/\CC$ be a smooth, connected, quasi-projective variety and let $L$ be a rank 2 $\CC$-local system on $X$ such that
\begin{itemize}
\item $L$ has trivial determinant,
\item $L$ has quasi-unipotent monodromy along the divisor at $\infty$,
\item $L$ has Zariski-dense monodromy inside of $SL_{2}(\mathbb{C})$,
and
\item $L$ is rigid.
\end{itemize}
Then $L$ comes from a family of abelian varieties: there exists an abelian scheme $\pi\colon A_{X}\rightarrow X$ such that
\[
R^{1}\pi_{*}\CC\cong{\displaystyle \bigoplus_{\sigma\in\Sigma}}(^{\sigma}L)^{m}
\]
where $\Sigma\subset\text{Aut}(\mathbb{C})$ is a finite subset of automorphisms of $\mathbb{C}$ containing the identity, $^{\sigma}L$ is the local system obtained by applying $\sigma$ to the matrices in the associated $SL_{2}(\CC)$ representation, and $m\in\mathbb{N}$.
\end{thm*}

When $X$ is projective, $L$ being rigid means it yields an isolated (though not necessarily reduced) point in the character variety associated to $\pi_{1}(X)$. For general quasi-projective $X$, the notion of rigidity involves a character variety that remembers Jordan blocks of the various monodromies around $\infty$, see \cite[Section 6]{corlette2008classification} or \cite[Section 2]{esnault2017cohomologically} for a precise definition. Corlette-Simpson's theorem verifies the rank 2 case of a conjecture of Simpson that roughly states: rigid semi-simple $\CC$-local systems on smooth complex varieties are motivic. (When $X$ is proper, this is \cite[Conjecture 4]{simpson1990nonabelian}.)
\end{rem}

\begin{rem}
\label{Remark:geometric_origin}It is conjectured that if $X/\Fq$ is an irreducible, smooth variety and $L$ is an irreducible lisse $\Qlbar$-sheaf with trivial determinant, then $L$ is ``of geometric origin'' up to a Tate twist (Esnault-Kerz attribute this to Deligne \cite[Conjecture 2.3]{esnault2011notes}, see also \cite[Question 1.4]{drinfeld2012conjecture}). More precisely, given such an $L$, it is conjectured that there exists an open dense subset $U\subset X$, a smooth projective morphism $\pi:Y_{U}\rightarrow U$, an integer $i$, and a rational number $j$ such that $L|_{U}$ is a sub-quotient (or even a summand) of $R^{i}\pi_{*}\Qlbar(j)$.
\end{rem}

\begin{rem}
Compared to Corlette-Simpson's theorem, there are two fewer hypotheses in Conjecture \ref{Conjecture:R2}: there is no ``quasi-unipotent monodromy at $\infty$'' condition and there is no rigidity condition. The former is automatic by Grothendieck's quasi-unipotent monodromy theorem \cite[Appendix]{serre1968good}.

As for the latter: the local systems showing up in Simpson's work are representations of a \emph{geometric fundamental group}, while the local systems occuring in Conjecture \ref{Conjecture:R2} are representations of an \emph{arithmetic fundamental group.} Let $X/\Fq$ be a smooth variety. Then it follows from  \cite{deligne2012finitude} that there are only finitely many isomorphism classes of irreducible lisse $\Qlbar$-sheaves that have trivial determinant, bounded rank, and bounded ramification \cite[Theorem 1.1]{esnault2012finiteness}.  For this reasons, the authors view such local systems as morally rigid.
\end{rem}

\begin{rem}
A further difference between Conjecture \ref{Conjecture:R2} and Corlette-Simpson's theorem is the intervention of an open set $U\subset X$ in the positive characteristic case. If $U\subset X$ is a Zariski open subset of a smooth and irreducible $\mathbb{C}$-variety with complement of codimension at least 2, then any map $U\rightarrow\mathscr{A}_{g}\otimes\mathbb{C}$ extends (uniquely) to a map $Y\rightarrow\mathscr{A}_{g}\otimes\mathbb{C}$ \cite[Corollaire 4.5]{grothendieck1966theoreme}. This extension property is not true in characteristic $p$; one may construct counterexamples using a family of supersingular abelian surfaces over $\mathbb{A}^{2}\backslash\{(0,0)\}$ \cite[Remarques 4.6]{grothendieck1966theoreme}. On the other hand, we do not know of a single example that requires the intervention of an open subset $U\subsetneq X$ in Conjecture \ref{Conjecture:R2}.
\end{rem}

\begin{rem}
There is recent work of Snowden-Tsimerman that characterizes those rank 2 $\overline{\mathbb Q}_l$ sheaves on a curve over a number field that come from a family of elliptic curves \cite{snowden2018constructing}. This work was very inspiring for us, but the techniques used there are rather different from those used here. In particular, they use Drinfeld's result modolu $p$ for infinitely many $p$, together with a Hilbert scheme argument based on the boundedness of maps between projective hyperbolic curves. Our techniques instead use deformation theory and algebraization.
\end{rem}
We briefly summarize the sections.
\begin{itemize}
\item Section \ref{Section:Coefficient_Objects} sets out definitions and conventions for lisse $l$-adic sheaves and (overconvergent) $F$-isocrystals. Following Kedlaya, if $X/\Fq$ is a smooth variety, we define a \emph{coefficient object} to be either a lisse $l$-adic sheaf or an overconvergent $F$-isocrystal.
\item Section \ref{Section:Compatible_systems} briefly describes the companions conjecture.
\item Section \ref{Section:BT groups} reviews what we need from the theory of Barsotti-Tate (a.k.a. $p$-divisible) groups.
\item Section \ref{Section:Grothendieck} relates extending abelian schemes with extending their $p$-divisible group. The main result is Corollary \ref{Corollary:grothendieck}.
\item Section \ref{Section:Lefschetz I} recalls a result of Hartshorne that allows us to globalize formal maps.
\item Section \ref{Section:Lefschetz 2} puts all of the ingredients together and contains a proof of of Theorem \ref{Theorem:main_lefschetz}.

\end{itemize}

The authors hope this work helps further reveal the geometric content of $p$-adic coefficient objects; more specifically that they are analogous to variations of Hodge structures. Deligne's Conjecture \ref{Conjecture:Deligne} was presumably formulated with the hope that such local systems are of geometric origin; here, we deduce that certain local systems are of geometric origin from the existence of $p$-adic companions. 
\begin{acknowledgement*}
This work was born at CIRM (in Luminy) at ``$p$-adic Analytic Geometry and Differential Equations''; the authors thank the organizers. R.K. warmly thanks Tomoyuki Abe, Ching-Li Chai, Marco d'Addezio, Philip Engel, H\'el\`ene Esnault, and especially Johan de Jong, with whom he had stimulating discussions on the topic of this article. R.K. also thanks Daniel Litt, who patiently and thoroughly explained several algebraization techniques; in particular, Litt explained the proof of  Lemma \ref{Lemma_BT_extending_iso}. R.K. thanks the arithmetic geometry group at FU Berlin for a lively and intellectually rich atmosphere. R.K. thanks Dino Lorenzini for comments on an earlier draft of this article. R.K. gratefully acknowledges financial support from the NSF under Grants No. DMS-1605825 and No. DMS-1344994 (RTG in Algebra, Algebraic Geometry and Number Theory at the University of Georgia).
\end{acknowledgement*}

\section{Notation and Conventions}
\begin{itemize}
\item The field with $p$ elements is denoted by $\Fp$ and $\FF$ denotes a fixed algebraic closure.
\item A variety $X/k$ is a geometrically integral scheme of finite type.
\item If $X$ is a scheme, then $X^{\circ}$ is the set of closed points.
\item If $k$ is a field, $l$ denotes a prime different than char$(k)$.
\item If $E$ is a number field, then $\lambda$ denotes an arbitrary prime of $E$.
\item An $\lambda$-adic local field is a \emph{finite }extension of $\QQ_{\lambda}$.
\item If $L/K$ is a finite extension of fields and $\mathcal{C}$ is a $K$-linear abelian category, then $\mathcal{C}_{L}$ is the base-changed category. If $M/K$ is an algebraic extension, then $\mathcal{C}_{M}$
is the 2-colimit of the categories $\mathcal{C}_{L}$ as $M$ ranges through the finite extensions of $K$ contained in $M$: $K\subset L\subset M$.
\item If $\calG\rightarrow S$ is a $p$-divisible (a.k.a. Barsotti-Tate) group, then $\DD(\calG)$ denotes the contravariant Dieudonn\'e crystal.
\item If $X/k$ is a smooth scheme of finite type over a perfect field $k$, then $\fisocd{X}$ is the category of overconvergent $F$-isocrystals on $X$. \item If $X$ is a noetherian scheme and $Z\subsetneq X$ is a non-empty closed subscheme, then $X_{/Z}$ is the formal completion of $X$ along $Z$.
\end{itemize}

\section{Coefficient Objects}\label{Section:Coefficient_Objects}

For a more comprehensive introduction to the material of this section, see the recent surveys of Kedlaya \cite{kedlaya2016notes,kedlayacompanions}; our notations are consistent with his (except for us $l\neq p$ and $\lambda$ denotes an arbitrary prime). Throughout this section, $k$ denotes a perfect field, $W(k)$ the ring of Witt vectors, $K(k)$ the field of fractions of $W(k)$, and $\sigma$ is the canonical lift of absolute Frobenius on $k$.

Let $X/\Fq$ be a normal variety and $\FF$ a fixed algebraic closure of $\Fq$. Denote by $\overline{X}$ the base change $X\times_{\Spec{\Fq}}\Spec{\FF}$. We have the following homotopy exact sequence
\[
0\rightarrow\pi_{1}(\overline{X})\rightarrow\pi_{1}(X)\rightarrow\text{Gal}(\mathbb{F}/\mathbb{F}_{q})\rightarrow 0
\]
(suppressing the implicit geometric point required to define $\pi_{1}$). The profinite group $\text{Gal}(\FF/\Fq)$ has a dense subgroup $\mathbb{Z}$ generated by the Frobenius. The inverse image of this copy of $\mathbb{Z}$ in $\pi_{1}(X)$ is called the
\emph{Weil Group $W(X)$ }\cite[1.1.7]{deligne1980conjecture}. The Weil group is given a topology where $\mathbb{Z}$ is discrete; this is not the subspace topology from $\pi_{1}(X)$.
\begin{defn}
\cite[1.1.12]{deligne1980conjecture} Let $X/\Fq$ be a normal variety and $K$ an $l$-adic local field. A \emph{(lisse) Weil sheaf of rank $r$ with coefficients in $K$ } is a continuous representation $W(X)\rightarrow GL_{r}(K)$. A \emph{(lisse) \'etale
sheaf of rank $r$ with coefficients in $K$} is a continuous representation $\pi_{1}(X)\rightarrow GL_{r}(K)$. We denote the category of Weil sheaves with coefficients in $K$ by $\weil{X}_{K}$.
\end{defn}

Every lisse \'etale sheaf yields a lisse Weil sheaf. Conversely, any lisse Weil sheaf with finite determinant is the restriction to $W(X)$ of an \'etale sheaf \cite[1.3.4]{deligne1980conjecture}.

\begin{defn}
\label{Definition:extend_scalars}Let $\mathcal{C}$ be a $K$-linear additive category, where $K$ is a field \cite[Tag 09MI]{stacks-project}. Let $L/K$ be a finite field extension. We define the base-changed category $\calC_{L}$ as follows:
\begin{itemize}
\item Objects of $\calC_{L}$ are pairs $(M,f)$, where $M$ is an object of $\mathcal{C}$ and $f\colon L\rightarrow \End_{\calC}M$ is a homomorphism of $K$-algebras. We call such an $f$ an $L$-structure on $M$.
\item Morphisms of $\calC_{L}$ are morphisms of $\calC$ that are compatible with the $L$-structure.
\end{itemize}
\end{defn}

\begin{fact}
\label{Fact:extension_of_scalars}Let $\calC$ be a $K$-linear abelian category and let $L/K$ be a finite field extension. Then there are functors
\[
Ind_{K}^{L}\colon\mathcal{C}\leftrightarrows\mathcal{C}_{L}\colon Res_{K}^{L}
\]
called induction and restriction. Restriction is right adjoint to induction. Both functors send semi-simple objects to semi-simple objects. If $\calC$ has the structure of a Tannakian category, then so does $\calC_L$. In this case, $Ind^L_K$ preserves rank and $Res^L_K$ multiplies the rank by $[L:K]$.
\end{fact}

\begin{proof}
The restriction functor is given by forgetting the $L$-structure. For a description of induction and restriction functors, see \cite[Section 3]{krishnamoorthy2017rank}. The fact about semi-simple objects follows immediately from \cite[Corollary 3.12]{krishnamoorthy2017rank}. For the fact that $\calC_L$ is Tannakian, see \cite{deligne2014semi}.
\end{proof}
Given an object $M$ of $\calC$, we will sometimes write $M_{L}$ or $M\otimes_{K}L$ for $Ind_{K}^{L}M$.

The category of $\Ql$-Weil sheaves, $\weil{X}$, is naturally an $\Ql$-linear neutral Tannakian category, and $\weil{X}_{K}\cong(\weil{X})_{K}$, where the latter denotes the ``based-changed category'' as above. We define $\weil{X}_{\Qlbar}$ as the 2-colimit of
$\weil{X}_K$ as $K\subset\Qlbar$ ranges through the finite extensions of $\Ql$ inside of $\Qlbar$. Alternatively $\weil{X}_{\Qlbar}$ is the category of continuous, finite dimensional representations of $W(X)$ in $\Qlbar$-vector spaces where $\Qlbar$ is equipped with the colimit topology.
\begin{defn}\label{Definition:char_poly_lisse}
Let $X/\Fq$ be a smooth variety and let $x\in X^{\circ}$ be a closed point of $X$. Given a Weil sheaf $L$ on $X$ with coefficients in $\Qlbar$, by restriction we get a Weil sheaf $L$ on $x$. Denote by $Fr_{x}$ the geometric Frobenius at $x$. Then $Fr_{x}\in W(x)$ and we define $P_{x}(L,t)$, \emph{the characteristic polynomial of $L$ at $x$,} to be
$$P_x(L,t):=\det(1-Fr_{x}t|L_{x}).$$
\end{defn}

Now, let $X/k$ be a scheme of finite type over a perfect field. Berthelot has defined the absolute crystalline site on $X$: for a reference, see \cite[TAG 07I5]{stacks-project}. (We implicitly take the crystalline site with respect to $W(k)$ without further comment; in other words, in the formulation of the Stacks Project, $S=\Spec{W(k)}$ with the canonical divided power structure.) Let $Crys(X)$ be the category of crystals in \emph{finite locally free $\mathcal{O}_{X/W(k)}$-modules}. To make this more concrete, we introduce the following notation. A PD test object is a triple 
\[
(R,I,(\gamma_{i}))
\]
where $R$ is a $W(k)$ algebra with $I$ a nilpotent ideal such that $\Spec{R/I}$ ``is'' a Zariski open of $X$, and $(\gamma_{i})$ is a PD structure on $I$. Then a crystal in finite locally free modules $M$ on $X$ is a rule with input a PD test object $(R,I,(\gamma_{i}))$ and output a finitely generated projective $R$ module
\[
M_{R}
\]
that is functorial: the pullback maps with respect to morphisms of PD test-objects are isomorphisms. In this formulation, the crystalline structure sheaf $\mathcal{O}_{X/W(k)}$ has the following description: on input $(R,I,(\gamma_{i}))$, the sheaf $\mathcal{O}_{X/W(k)}$ has as output the ring $R$. 

By functoriality of the crystalline topos, the absolute Frobenius $Frob\colon X\rightarrow X$ gives a functor $Frob^{*}\colon Crys(X)\rightarrow Crys(X)$. 
\begin{defn}
\cite[TAG 07N0]{stacks-project} A (non-degenerate) $F$\emph{-crystal} on $X$ is a pair $(M,F)$ where $M$ is a crystal in finite locally free modules over the crystalline site of $X$ and $F\colon Frob^{*}M\rightarrow M$ is an injective map of crystals.
\end{defn}

We denote the category of $F$-crystals by $FC(X)$; it is a $\ZZ_{p}$-linear category with an internal $\otimes$ but without internal homs or duals in general. There is a object $\ZZ_{p}(-1)$, given by the pair $(\mathcal{O}_{X/W(k)},p)$. We denote by $\ZZ_{p}(-n)$ the $n$th tensor power of $\mathbb{Z}_{p}(-1)$.
\begin{notation}
\cite[Definition 2.1]{kedlaya2016notes} Let $X/k$ be a scheme of finite type over a perfect field. We denote by $\fisoc{X}$ the category of \emph{(convergent) $F$-isocrystals on $X$}.
\end{notation}

$\fisoc{X}$ is a $\Qp$-linear Tannakian category. Denote by $\Qp(-n)$ the image of $\ZZ_{p}(-n)$ and by $\Qp(n)$ the dual of $\Qp(-n)$. There is a notion of the rank of an $F$-isocrystal that satisfies that expected constraints given by $\otimes$ and $\oplus$. Unfortunately, $\fisoc{X}$ is not simply the isogeny category of $FC(X)$; however, it is the isogeny category of the category of $F$-crystals in \emph{coherent modules} (rather than finite, locally free modules). There is a natural functor $FC(X)\rightarrow\fisoc{X}$ \cite[2.2]{kedlaya2016notes}. 

For general smooth $X/k$ over a perfect field, it seems that there are two $p$-adic analogs of a lisse $l$-adic sheaf: a (convergent) $F$-isocrystal and an \emph{overconvergent $F$-isocrystal}. For our purposes (and following \cite{crew1992f}), overconvergent $F$-isocrystals are a better analog.
\begin{notation}
\cite[Definition 2.7]{kedlaya2016notes} Let $X/k$ be a smooth variety over a perfect field. We denote by $\fisocd{X}$ the category of \emph{overconvergent $F$-isocrystals on $X$}.
\end{notation}

The category $\textbf{F-Isoc}^{\dagger}(X)$ is a $\Qp$-linear Tannakian category. There is a natural forgetful functor
\[
\fisocd{X}\rightarrow\fisoc{X}
\]
which is fully faithful in general \cite[Theorem 1.1]{kedlaya2001full} and an equivalence of categories when $X$ is proper. 
\begin{notation}
For any finite extension $L/\Qp$, we denote by $\fisocd{X}_{L}$ the base-changed category a.k.a. \emph{overconvergent $F$-isocrystals with coefficients in $L$.}
\end{notation}
\begin{defn}
Let $X/k$ be a smooth scheme over a perfect field. Let $L$ be a $p$-adic local field and let $\calE$ be an object of $\fisocd{X}_L$. We say $\calE$ is \emph{effective} if the object $Res_{\Qp}^{L}\calE$ is in the essential image of the functor
\[
FC(X)\rightarrow\fisocd{X}
\]
\end{defn}

Being effective is equivalent to the existence of a ``locally free lattice stable under $F$''. See Lemma \ref{Lemma:Katz} for the basic result on effectivity.

Let $k$ be a perfect field. Then the category of $F$-crystals on $\text{Spec}(k)$ is equivalent to the category of finite free $W(k)$-modules $M$ together with a $\sigma$-linear injective map
\[
F\colon M\rightarrow M.
\]
(Morphisms in this category are defined to be $W(k)$-linear maps that commute with the $F$s.) Similarly, the category of $F$-isocrystals on $\text{Spec}(k)$ is equivalent to the category of finite dimensional $K(k)$-vector space $V$ together with a $\sigma$-linear bijective map
\[
F\colon V\rightarrow V,
\]
where morphisms are defined in the obvious way.
The rank of $(V,F)$ is the rank of $V$ as a vector space. Let $L/\Qp$ be a finite extension. Then $\fisoc{k}_L$ is equivalent to the following category: objects are pairs $(V,F)$ where $V$ is a finite free $K(k)\otimes_{\Qp}L$ module and $F$ is a $\sigma\otimes1$-linear bijective map
\[
F\colon V\rightarrow V
\]
and morphisms are maps of $K(k)\otimes_{\Qp}L$-modules that commute with $F$ \cite[Proposition 5.12]{krishnamoorthy2017rank}. Note that $K(k)\otimes_{\Qp}L$ is not necessarily a field. There is also a direct-sum decomposition of abelian categories:
\[
\fisoc{k}_{L}\cong\bigoplus_{\lambda\in\QQ_{\geq0}}\fisoc{k}_{L}^{\lambda}
\]
which is inherited from the analogous decomposition of $\fisoc{k}$. Here, $\fisoc{k}_{L}^{\lambda}$ is the (thick) abelian sub-category with objects isoclinic of slope $\lambda$. 

\begin{notation}
We denote by $\fisocd{X}_{\Qpbar}$ the 2-colimit of the base-changed categories over all finite extensions $\Qp\subset L\subset\Qpbar$, via the functors $Ind_{\Qp}^{L}$. This is the category of \emph{overconvergent $F$-isocrystals on $X$ with coefficients in $\Qpbar$}.
\end{notation}

\begin{rem}
When $k$ is a perfect field, $\fisoc{k}_{\Qpbar}$ has the following description. Objects are pairs $(V,F)$ where $V$ is a finite free $K(k)\otimes_{\Qp}\Qpbar$-module and $F\colon V\rightarrow V$ is a bijective, $\sigma\otimes1$-linear map. Morphisms are $K(k)\otimes_{\Qp}\Qpbar$-linear maps that commute with $F$. 
\end{rem}

Let $L/\Qp$ be an algebraic extension, let $k\cong\FF_{p^{d}}$ and let $(V,F)$ be an object of $\fisoc{k}_L$. Then $F^{d}$ acts as a linear map on $V$. We let $P((V,F),t)$, the \emph{characteristic polynomial of Frobenius}, be $\det(1-(F^{d})t|V)$. By \cite[Proposition 6.1]{krishnamoorthy2017rank}
\[
P((V,F),t)\in L[t]
\]

\begin{defn}
Let $X/\Fq$ be a smooth variety, let $L/\Qp$ be an algebraic extension, and let $\calE\in\fisocd{X}_L$.
Let $x\in X^{\circ}$ with residue field $\FF_{p^{d}}$. Define $P_{x}(\calE,t)$, \emph{the characteristic polynomial of $\calE$
at $x$}, to be $\det(1-(F^{d})t|\calE_{x})$.
\end{defn}

\begin{notation}
\cite[Notation 1.1.1]{kedlayacompanions} Let $X/\Fq$ be a smooth connected variety. A \emph{coefficient object} is an object either of $\weil{X}_{\Qlbar}$ or of  $\fisocd{X}_{\Qpbar}$. We informally call the former the \emph{\'etale case} and the latter
the \emph{crystalline case}. 
\end{notation}

Given an algebraic extension $\Ql \subset K\subset\Qlbar$, objects of $\weil{X}_K$ may be considered as \'etale coefficient objects. Similarly, given an algebraic extension $\Qp\subset K\subset \Qpbar$, then objects of $\fisocd{X}_K$ may be considered as
crystalline coefficient objects via the induction functor.
\begin{defn}
Let $\calE$ be an \'etale (resp. crystalline) coefficient object. Let $K$ be a subfield of $\Qlbar$ (resp. of $\Qpbar$) containing $\Ql$ (resp. containing $\Qp$). The following three equivalent phrases
\begin{itemize}
\item $\calE$ has \emph{coefficients in} $K$
\item $\calE$ has \emph{coefficient field} $K$
\item $\calE$ is a $K$\emph{-coefficient object}
\end{itemize}
mean that $\calE$ may be descended to $\weil{X}_K$ (resp. $\fisocd{X}_K$).
\end{defn}

Note that $\calE$ having coefficient field $L$ does not preclude $\calE$ from having coefficient field $K$ for some sub-field $K\subset L$. 

\begin{defn}
\cite[Definition 1.1.5]{kedlayacompanions} Let $X/\Fq$ be a smooth variety and let $\calF$ be a coefficient object on $X$ with coefficients in $\Qlambar$. We say $\calF$ is \emph{algebraic} if $P_{x}(\mathcal{F},t)\in\overline{\mathbb{Q}}[t]\subset \Qlambar[t]$ for all $x\in X^{\circ}$. Let $E\subset\Qlambar$ be a number field. We say $\calF$ is $E$\emph{-algebraic } if $P_{x}(\calF,t)\in E[t]$ for all $x\in X^{\circ}$.
\end{defn}

Colloquially, semi-simple coefficient objects are determined by the characteristic polynomials of Frobenius at all closed points. In the \'etale case, this is a consequence of the Brauer-Nesbitt theorem and the Chebotar\"ev density theorem. In the crystalline case, the argument is more subtle and is due to Tsuzuki \cite[A.4.1]{abe2013langlands}. See also a proof of this fact in \cite{hartlpal} which is closer in spirit to the original proof and uses a $p$-adic variant of the Chebotar\"ev density theorem.
\begin{thm}
\label{Theorem:Brauer_Nesbitt_coefficients}\cite[Theorem 3.3.1]{kedlayacompanions} Let $X/\Fq$ be a smooth variety. Let $\calF$ be a semi-simple coefficient object. Then $\calF$ is determined, up to isomorphism, by $P_{x}(\calF,t)$ for all $x\in X^{\circ}$.
\end{thm}

\section{Compatible Systems and Companions}\label{Section:Compatible_systems}
\begin{defn}
\label{Definition:Companion}Let $X/\Fq$ be a smooth variety. Let $\calE$ and $\calE'$ be algebraic coefficient objects on $X$ with coefficients in $\Qlambar$ and $\Qlambarp$ respectively. Fix a field isomorphism
$\iota\colon\Qlambar\rightarrow\Qlambarp$. We say $\mathcal{E}$ and $\mathcal{E}'$ are \emph{$\iota$-companions} if $$\iota (P_{x}(\calE,t))=P_{x}(\calE',t)$$
for all $x\in X^{\circ}$. We say that $\mathcal{E}$ and $\mathcal{E}'$ are companions if there exists an isomorphism $\iota\colon \Qlambar\rightarrow \Qlambarp$ that makes them $\iota$-companions.
\end{defn}
In other words, two coefficient objects are $\iota$-companions if, under $\iota$, the characteristic polynomials of Frobenius match up.

\begin{rem}
\label{Remark:companions_l-adic}For convenience, we spell out the notion of companions for lisse $l$-adic sheaves. Let $X/\Fq$ be a smooth variety and let $L$ and $L'$ be lisse $\Qlbar$ and $\overline{\QQ}_{l'}$
sheaves respectively. (Here $l$ may equal $l'$.) We say they are companions if there exists a field isomorphism $\iota\colon \Qlbar\rightarrow\overline{\QQ}_{l'}$ such that for every closed point $x$ of $X$, there is an equality of polynomials of Frobenius at $x$ (see Definition \ref{Definition:char_poly_lisse}):
\[
\iota(P_{x}(L,t))=P_{x}(L',t)\in\overline{\QQ}_{l'}[t],
\]
and furthermore this polynomial is in $\overline{\QQ}[t]$. 
\end{rem}

\begin{rem}
Note that the $\iota$ in Definition \ref{Definition:Companion} and Remark \ref{Remark:companions_l-adic} does not reference the topology of $\Qpbar$ or $\Qlbar$. In particular, $\iota$ need not be continuous and in fact cannot be continuous if $l\neq p$.
\end{rem}

\begin{defn}
Let $X/\Fq$ be a smooth variety and let $E$ be a number field. Then an \emph{$E$-compatible system} is a system of lisse $E_{\lambda}$-sheaves $(\calF_{\lambda})_{\lambda\nmid p}$
over primes $\lambda\nmid p$ of $E$ such that for every $x\in X^{\circ}$
\[
P_{x}(\calF_{\lambda},t)\in E[t]\subset E_{\lambda}[t]
\]
and this polynomial is independent of $\lambda$. 
\end{defn}

\begin{defn}
\label{Definition:complete_set_of_companions}Let $X/\Fq$ be a smooth variety and let $\lambda$ be a rational prime. Let $(\calE_{v})_{v\in\Lambda}$ be a collection of $\Qlambar$ coefficient objects on $X$, indexed by a finite set $\Lambda$. We say they form a \emph{complete set of $\lambda$-adic companions} if 
\begin{itemize}
\item for every $v\in\Lambda$ and for every $\iota\in\text{Aut}_{\QQ}(\Qlambar)$, there exists $v'\in\Lambda$ such that $\mathcal{E}_{v}$ and $\mathcal{E}_{v'}$ are $\iota$-companions, and
\item For every $v,v'\in\Lambda$, $\calE_v$ and $\calE_{v'}$ are $\iota$-companions for some $\iota\in\text{Aut}_{\QQ}(\Qlambar)$.
\end{itemize}
\end{defn}

The following definition of a complete $E$-compatible system involves \emph{all possible} $\iota$-companions.
\begin{defn}
Let $X/\Fq$ be a smooth variety and let $E$ be a number field. A \emph{complete $E$-compatible system} $(\calF_{\lambda})$ is an $E$-compatible system together with, for each prime $\lambda$ of $E$ over $p$, an object 
\[
\calF_{\lambda}\in\fisocd{X}_{E_{\lambda}}
\]
such that the following two conditions hold.
\begin{enumerate}
\item For every place $\lambda$ of $E$ and every $x\in X^{\circ}$, the polynomial $P_{x}(\calF_{\lambda},t)\in E[t]\subset E_{\lambda}[t]$
is independent of $\lambda$.
\item For a prime $r$ of $\mathbb{Q}$, let $\Lambda$ denote the set of primes of $E$ above $r$. Then $(\calF_{\lambda})_{\lambda\in\Lambda}$ form a complete set of $r$-adic companions for every $r$.
\end{enumerate}
\end{defn}
In particular, let $(\calF_{\lambda})$ be a complete $E$-compatible system and let $L$ be a lisse $\Qlbar$-adic sheaf such that $L$ is isomorphic an object of $(\calF_{\lambda})$. Let $\iota\colon \Qlbar\rightarrow \bar{\QQ}_{\lambda'}$. Then the $\iota$-companion to $L$ exists and is isomorphic to an object in my list $(\calF_{\lambda})$. The most basic example of an $E$-compatible system is the following.

\begin{example}\label{Example:hilbert_complete_companions}
Let $E/\QQ$ be a totally real number field of degree $g$ and let $\calM$ be a Hilbert modular variety parametrizing principally polarized abelian $g$-folds with multiplication by a given order $\mathcal{O}\subset E$ and some level structure. For most primes $\mathfrak{p}$
of $E$, $\mathcal{M}$ has a smooth integral canonical model $\tilde{\calM}$ over $\mathcal{O}_{\mathfrak{p}}$; moreover there is a universal abelian scheme $\tilde{\mathcal{A}}\rightarrow\tilde{\calM}$. Let $\pi\colon A\rightarrow M$ denote the special fiber of such a smooth canonical model together with the induced abelian scheme. Then $E\hookrightarrow \End_{M}(A)\otimes\QQ$ and in particular
the local system $R^{1}\pi_{*}(\Ql)$ admits an action
by
\[
E\otimes\Ql\cong{\displaystyle \prod_{v|l}E_{v}}
\]
Here $v$ ranges over the primes of $E$ over $l$. Let $e_{v}$ denote the idempotent projecting $E\otimes\Ql$ onto $E_{v}$. Then $L_{v}:=e_{v}R^{1}\pi_{*}(\Ql)$ is a rank 2 lisse $\mathbb{Q}_{l}$-sheaf with an action of $E_{v}$, in other words
a rank 2 lisse $E_{v}$-sheaf. It follows from the techniques of \cite[11.9,11.10]{shimura1967} that the $L_{v}$ are all companions. In fact, $(L_{v})_{v|l}$ is a complete set of $l$-adic companions. If $[E:\mathbb{Q}]>1$ and $l$ splits non-trivially in $E$, then these will in general be mutually non-isomorphic lisse $l$-adic sheaves. By ranging over all primes of $\mathbb{Q}$ (using relative crystalline cohomology at $p$) we obtain a complete $E$-compatible system. 
\end{example}

We recall a conjecture of Deligne from Weil II \cite[Conjecture 1.2.10]{deligne1980conjecture}. 
\begin{conjecture}
\label{Conjecture:Deligne}Let $X/\FF_{p^d}$ be a normal variety with a geometric point $\overline{x}\rightarrow X$. Let $l\neq p$ be a prime. Let $\mathcal{L}$ be an absolutely irreducible $l$-adic local system with finite determinant on $X$. The choice
of $\overline{x}$ allows us to think of this as a representation $\rho_{l}\colon\pi_{1}(X,\overline{x})\rightarrow GL(n,\Qlbar)$.
Then

\begin{enumerate}
\item $\rho_{l}$ is pure of weight 0.
\item There exists a number field $E$ such that for all closed points $x$ of $X$, the polynomial $P_{x}(\mathcal{L},t)$ has all of its coefficients in $E$. In particular, the eigenvalues of $\rho_{l}(F_{x})$ are all algebraic numbers. 
\item For each place $\lambda\nmid p$, the roots $\alpha$ of $P_{x}(\mathcal{L},t)$ are $\lambda$-adic units in $\overline{E}_{\lambda}$.
\item For each $\lambda|p$, the $\lambda$-adic valuations of the roots $\alpha$ satisfy
\[
|\frac{v(\alpha)}{v(Nx)}|\leq\frac{n}{2}
\]
where $Nx$ is the size of the residue field of $x$.
\item After possibly replacing $E$ by a finite extension, for each $\lambda\nmid p$ there exists a $\lambda$-adic local system $\rho_{\lambda}:\pi_{1}(X,\overline{x})\rightarrow GL(n,E_{\lambda})$ that is compatible with $\rho_{l}$.
\item After possibly replacing $E$ by a finite extension, for each $\lambda|p$, there exists a crystalline companion to $\rho_{l}$.
\end{enumerate}
\end{conjecture}

The following conjecture may be seen as a refinement to (2), (5), and (6) of Conjecture \ref{Conjecture:Deligne}.
\begin{conjecture}
\label{Conjecture:Companions}(Companions) Let $X/\Fq$ be a smooth variety. Let $\calF$ be an irreducible coefficient object on $X$ with algebraic determinant. Then there exists a number field $E$ such that $\calF$ fits into a \textbf{complete}
$E$-compatible system. 
\end{conjecture}

The companions conjecture is surprising for the following reason: an $l$-adic local system is simply a continuous homomorphism from
$\pi_{1}(X)$ to $GL_{n}(\Qlbar)$, and the topologies on $\Ql$ and $\QQ_{l'}$ are completely different. Deligne likely made his conjecture out of the hope that such local systems were \emph{of geometric origin}. We summarize what is known
about Conjecture \ref{Conjecture:Deligne} and Conjecture \ref{Conjecture:Companions}.

By work of Deligne, Drinfeld, and Lafforgue, if $X$ is a curve all such local systems are of geometric origin (in the sense of subquotients). Moreover, in this case Chin has proved Part 5 of the conjecture \cite[Theorem 4.1]{chin2004independence}.
Abe has recently constructed a sufficiently robust theory of $p$-adic cohomology to prove a $p$-adic Langlands correspondence and hence answer affirmatively part 6 of Deligne's conjecture when $X$ is a curve \cite{abe2013langlands,abe2011langlands}.
\begin{thm}
\label{Theorem:Abe_correspondence} (Abe, Lafforgue) Let $C/\Fq$ be a smooth curve. Then both Deligne's conjecture and the companions conjecture are true for $C$.
\end{thm}

In higher dimension, much is known but Conjecture \ref{Conjecture:Companions} remains open in general. See \cite{esnault2012finiteness} for a precise chronology of the following theorem, due to Deligne and Drinfeld. 
\begin{thm}
\label{Theorem:Deligne_Drinfeld_l-adic}\cite{deligne2012finitude,drinfeld2012conjecture}
Let $X/\Fq$ be a smooth variety. Let $l\neq p$ be a prime. Let $\mathcal{L}$ be an absolutely irreducible $l$-adic local system with finite determinant on $X$. Then (1), (2), (3), and (5) of Conjecture \ref{Conjecture:Deligne} are true.
\end{thm}

\begin{thm}
(Deligne, Drinfeld, Abe-Esnault, Kedlaya) Let $X/\Fq$ be a smooth variety and let $\calF$ be a coefficient object that is absolutely irreducible and has finite determinant. Then for any $l\neq p$, all $l$-adic companions exist.
\end{thm}

\begin{proof}
This follows from Theorem \ref{Theorem:Deligne_Drinfeld_l-adic} together with either \cite[Theorem 4.2]{abe2016lefschetz} or \cite[Theorem 0.4.1]{kedlayacompanions}.
\end{proof}
\begin{rem}
The ``$p$-companions'' part of the conjecture is \emph{not known} if $\dim C>1$. Given an irreducible lisse $\Qlbar$ sheaf $L$ with trivial determinant and a (necessarily non-continuous) isomorphism $\iota\colon \Qlbar\rightarrow \Qpbar$,
it is \textbf{completely unknown} how to associate a crystalline $\iota$-companion to $L$.

We remark that Part (4) of the Conjecture \ref{Conjecture:Deligne} is not tight even for $n=2$.
\end{rem}

\begin{thm}(Abe-Lafforgue)\label{Theorem:rank_2_slope_bounds} Let $X/\Fq$ be a smooth variety and let $\mathcal{E}$ be an absolutely irreducible rank 2 coefficient object on $C$ with finite determinant. Then for all $x\in C^{\circ}$, the eigenvalues $\alpha$ of $F_{x}$ satisfy
\[
|\frac{v(\alpha)}{v(Nx)}|\leq\frac{1}{2}
\]
for any $p$-adic valuation $v$.
\end{thm}

\begin{proof}
Suppose $X$ is a curve. Then when $\mathcal{E}$ is an $l$-adic coefficient object, this is \cite[Theorem VII.6.(iv)]{lafforgue2002chtoucas}. For a crystalline coefficient object, we simply apply Theorem \ref{Theorem:Abe_correspondence}
to construct an $l$-adic coefficient object.

More generally, \cite[Proposition 2.17]{drinfeld2012conjecture} and \cite[Theorem 0.1]{abe2016lefschetz} function as ``Lefschetz theorems'' (see \cite{esnault2017survey}) and allow us to reduce to the case
of curves.
\end{proof}
\begin{rem}
A refined version of part (4) of Conjecture \ref{Conjecture:Deligne} has been resolved for curves by V. Lafforgue \cite[Corollaire 2.2]{lafforgue2011estimees} using automorphic forms. A $p$-adic variant (which avoids automorphic techniques) was shown directly by Drinfeld-Kedlaya \cite[Theorem 1.1.5]{drinfeld2016slopes}. They prove that for an indecomposable crystalline coefficient object, the generic consecutive slopes do not differ by more than 1, reminiscent of Griffith's transversality. Kramer-Miller has recently given another proof of the $p$-adic variant \cite{kramermiller}.
\end{rem}

Finally, we make the following simple observation: under strong assumptions on the field of traces, the existence of companions with the same residue characteristic is automatic.
\begin{cor}
\label{Corollary:all_companions_Galois_twists}Let $X/\Fq$ be a smooth variety, $E$ a number field, and $\calF$ be an irreducible $E$-algebraic coefficient object with coefficients in $\Qpbar$ (resp. $\Qlbar$). If there is only one prime in $E$ lying above $p$ (resp. $l$), then all $p$-adic (resp. $l$-adic) companions exist.
\end{cor}

\begin{rem*}
The hypothesis of Corollary \ref{Corollary:all_companions_Galois_twists} are certainly satisfied if $p$ (resp. $l$) is either totally inert or totally ramified in $E$.
\end{rem*}
\begin{proof}
We assume $\mathcal{F}$ is crystalline; the argument in the \'etale case will be precisely analogous. Suppose $\calF\in\fisocd{X}_K$ where $K/\Qp$ is a $p$-adic local field; we might as well suppose $K$ is Galois over $\mathbb{Q}_{p}$. As there is a
unique prime $\mathfrak{p}$ above $p$ in $E$, there is a unique $p$-adic completion of $E$. We now explicitly construct all $p$-adic companions. For every $g\in\text{Gal}(K/\Qp)$, consider the object $^{g}\calF$: in terms of Definition \ref{Definition:extend_scalars}, if $\calF=(M,f)$, then $^{g}\calF:=(M,f\circ g^{-1})$. We claim this yields all $p$-adic companions.

Fix embeddings $E\hookrightarrow K\hookrightarrow\Qpbar$.
Then there is a natural map
\[
\text{Gal}(K/\Qp)\rightarrow\Hom_{\QQ\text{-alg}}(E,\Qpbar)
\]
given by precomposing with the fixed embeding. As there is a single $\mathfrak{p}$ over $p$, this map is surjective. The result follows from Theorem \ref{Theorem:Brauer_Nesbitt_coefficients}, and the fact that $P_{x}(^{g}\calE,t)=\ ^{g}P_{x}(\calE,t)$ \cite[Proposition 6.16]{krishnamoorthy2017rank}.
\end{proof}
\begin{rem}
In the crystalline setting of Corollary \ref{Corollary:all_companions_Galois_twists}, the Newton polygons of all $p$-adic companions of $\calF$ are the same at all closed points. This does not contradict the example
\cite[Example 2.2]{koshikawa2015overconvergent}; in his example,
$p$ splits completely in the reflex field of the Hilbert modular
variety, and the reflex field is the same as the field generated by
the characteristic polynomial of Frobenius elements over all closed
points.
\end{rem}

\section{Barsotti-Tate Groups and Dieudonn\'e Crystals}\label{Section:BT groups}

In this section, again let $X/k$ denote a finite type scheme over a perfect field.
\begin{defn}
A \emph{Dieudonn\'e crystal }on $X$ is a triple $(M,F,V)$ where $(M,F)$ is an $F$-crystal (in finite locally free modules) on $X$ and $V:M\rightarrow Frob^{*}M$ is a map of crystals such that $V\circ F=p$
and $F\circ V=p$. We denote the category of Dieudonn\'e Crystals on
$X$ by $DC(X)$.
\end{defn}

\begin{rem}
Let $U/k$ be an locally complete intersection morphism over a perfect field $k$. Then the natural forgetful functor $DC(U)\rightarrow FC(U)$ is fully faithful. This follows from the fact that $H_{cris}^{0}(U)$ is $p$-torsion free. This in turn follows from the fact that if $U=\text{Spec}(A)$ is affine, then the $p$-adic completion of a PD envelope of $A$ is $W(k)$-flat \cite[Lemma 4.7]{de1999crystalline}.
\end{rem}

\begin{notation}
\cite[1.4.1.3]{chai2014complex} Let $BT(X)$ denote the category of Barsotti-Tate (``$p$-divisible'') groups on $X$.
\end{notation}

For a thorough introduction to $p$-divisible groups and their contravariant Dieudonn\'e theory, see \cite{grothendieck1974groupes,berthelot1982theorie,chai2014complex}. In particular, there exist Dieudonn\'e functors: given a BT group $\calG/X$ one can construct a Dieudonn\'e crystal $\DD(\calG)$ over $X$. According to our conventions, $\mathbb{D}$ is the \emph{contravariant} Dieudonn\'e functor. Moreover, given a $p$-divisible group $\calG$ on $X$, there exists the (Serre) dual, which we denote by $\calG^t$ (see \cite[10.7]{chai2009moduli} or \cite[1.4.1.3]{chai2014complex}), whose constituent parts are obtained via Cartier duality of finite, locally free group schemes. 
\begin{defn}
\cite[3.3]{chai2014complex} Let $\calG$ and $\calH$ be Barsotti-Tate groups on $X$ with the same height. Then an \emph{isogeny} is a homomorphism
\[
\phi\colon\calG\rightarrow\calH
\]
whose kernel is represented by a finite, locally free group scheme over $X$. 
\end{defn}

\begin{defn}
\cite[1.4.3]{chai2014complex} Let $\calG$ be a Barsotti-Tate group on $X$. A \emph{quasi-polarization} is an isogeny
\[
\phi\colon\calG\rightarrow\calG^{t}
\]
that is skew-symmetric in the sense that $\phi^{t}=-\phi$ under the natural isomorphism $\calG\rightarrow(\calG^{t})^{t}$.
\end{defn}

\begin{notation}
Let $\calG$ be a Barsotti-Tate group on $X$. We denote by $QPol(\calG)$ the set of quasi-polarizations of $\calG$.
\end{notation}

The next theorem is a corollary of \cite[Main Theorem 1]{de1995crystalline}.
\begin{thm}
\label{Theorem:de Jong}(de Jong) Let $X/k$ be a smooth scheme over a perfect field of characteristic $p$. Then the category $BT(X)$ is anti-equivalent to $DC(X)$ via $\mathbb{D}$.
\end{thm}
The following lemma partially explains the relationship between slopes and effectivity.
\begin{lem}
\label{Lemma:Katz} Let $X/k$ be a smooth variety over a perfect field of characteristic $p$ and let $\calE$ be an $F$-isocrystal
on $X$. Then
\begin{itemize}
\item all of the slopes of $\mathcal{E}$ at all $x\in|X|$ are non-negative if and only if there exists a dense Zariski open $U\subset X$, with
complement of codimension at least 2, such that $\calE|_{U}$ is effective (i.e., there exists an $F$-crystal $(M,F)$ in finite, locally free modules on $U$ such that $(M,F)\otimes\Qp\cong\calE_U$.
\item Furthermore, all of the slopes of $\calE$ at all points $x\in |X|$ are between 0 and 1 if and only if there exists a dense Zariski open
$U'\subset X$, with complement of codimension at least 2, such that $\calE_{U'}$ comes from a Dieudonn\'e crystal.
\end{itemize}
\end{lem}

\begin{proof}
The ``only if'' of both statements follows from the following facts.
\begin{enumerate}
\item If $\calE$ is an $F$-isocrystal on $X$, then the Newton polygon is an upper semi-continuous function with locally constant right endpoint (a theorem of Grothendieck, see e.g. \cite[Theorem 2.3.1]{katz1979slope}).
\item The locus of points where the slope polygon does not coincide with it's generic value is pure of codimension 1 \cite[Theorem 3.12 (b)]{kedlaya2016notes}.
\end{enumerate}
The ``if'' of the first statement is implicit in \cite[Theorem 2.6.1]{katz1979slope}; see also \cite[Remark 2.3]{kedlaya2016notes}.

Finally, let $(M,F)$ be an $F$-crystal in finite, locally free modules on $U$ with all Newton slopes $\leq 1$. We explain the small modification to the proof of \cite[Theorem 2.6.1]{katz1979slope} to find an open set $U'\subset U$, with complementary codimension at least 2, and an isogenous $F$-crystal $(M',F)$ (in finite, locally free modules) on $U'$ that underlies a Dieudonn\'e crystal.

Set $V:=F^{-1}\circ p$. Then $V$ does not necessarily stabilize the lattice $M$. However, the pair $(M,V)$ has the structure of $\sigma^{-1}\!\mbox{-}F$-isocrystal in the language of \cite[p. 115]{katz1979slope}. (Fortunately, every result in Katz's paper is written for $\sigma^{a}\!\mbox{-}F$-crystals for all $a\neq 0$, not just the positive ones!) Consider the following (\`a priori infinite) sum, obviously stabilized by $V$:
$$\tilde{M}:=\displaystyle\sum_{n\geq 0} V^nM\subset \calE.$$
Then \cite[Page 152]{katz1979slope} implies that $\tilde{M}$ is crystal in \emph{coherent} modules on $U$. It follows from the fact that $FV=VF=p$ that $\tilde{M}$ is stabilized by $F$. Running the rest of the argument of \cite[Theorem 2.6.1]{katz1979slope}, we obtain an open $U'\subset U$ of complementary codimension at least 2 and a crystal $\calM'\subset \calE_{U'}$ in finite, locally free modules on $U'$ that is stabilized by $F$ and $V$. (The argument amounts to taking the double dual, which is automatically a reflexive sheaf, and noting that a reflexive sheaf on a regular scheme is locally free away from a closed subset of codimension at least 3. See also the following remark \cite[Page 154]{katz1979slope}.) Therefore, the triple $(\calM',F,V)$ is a Dieudonn\'e crystal on $U'$ as desired.
\end{proof}
\begin{rem}
The $F$-crystal constructed in Lemma \ref{Lemma:Katz} is \emph{not unique}.
\end{rem}
\begin{rem}It is \emph{not true} that if $(M,F)$ is an $F$-crystal with all Newton slopes in $[0,1]$, then $(M,F)$ underlies a Dieudonn\'e crystal (equivalently, it is not true that $(M,F)$ has all Hodge slopes 0 or 1). Here is an example of an $F$-crystal over $\Fp$ with Newton slopes all 1 but Hodge slopes 0 and 2: $M=\Zp^2$, and $F$ is given by the following matrix:
$$F:=
\begin{pmatrix}
0 & p^2 \\
1 & 0 
\end{pmatrix}
$$
\end{rem}
The following notion of \emph{duality} of Dieudonn\'e crystals is designed to be compatible with (Serre) duality of BT groups. In particular, the category $DC(X)$ admits a natural anti-involution. 
\begin{defn}
\cite[2.10]{chai2009moduli} Let $\scrM=(M,F,V)$ be a Dieudonn\'e crystal on $X$. Then the \emph{dual} $\scrM^{\lor}$ is a Dieudonn\'e crystal on $X$ whose underlying crystal is defined on divided power
test rings $(R,I,(\gamma_{i}))$ as
\[
\scrM^{\lor}(R,I,(\gamma_{i}))=\Hom_{R}(M(R),R)
\]
The operators $F$ and $V$ are defined as follows: if $(R,I,(\gamma_{i}))$
is a PD test object and $f\in\Hom_{R}(M(R),R)$, then
\[
F(f)(m)=f(Vm)
\]
\end{defn}

\[
V(f)(m)=f(Fm)
\]
It is an easy exercise to check that these rules render $\scrM^{\lor}$ a Dieudonn\'e crystal. 
\begin{rem}
\label{Remark:cartier_dual_dieudonne}If $\calG\rightarrow S$ is a BT group, then $\DD(\calG)^{\lor}\cong\DD(\calG^{t})$
. \cite[5.3.3.1]{berthelot1982theorie}.
\end{rem}

We record the following lemma, which is surely well-known, as we could not find a reference. It compares the dual of a Dieudonn\'e crystal with the dual of the associated $F$-isocrystal. 
\begin{lem}
\label{Lemma:two_notions_of_duality}Let $\mathscr{M}=(M,F,V)$ be a Dieudonn\'e crystal on $X$. Then the following two $F$-isocrystals on $X$ are naturally isomorphic:
\[
\scrM^{\lor}\otimes\QQ\cong(\scrM\otimes\QQ)^{*}(-1)
\]
\end{lem}

\begin{proof}
There is a natural perfect pairing to the Lefschetz $F$-crystal \[
\scrM\otimes\scrM^{\lor}\rightarrow\Zp(-1)
\]
as $FV=p$. There is also a perfect pairing
\[
(\scrM\otimes\QQ)\otimes(\scrM\otimes\QQ)^{*}\rightarrow\Qp
\]
to the ``constant'' $F$-isocrystal. Combining these two pairings
shows the result. 
\end{proof}

\section{A question of Grothendieck}\label{Section:Grothendieck}

In this section, we answer the question posed in \cite[4.9]{grothendieck1966theoreme}, see Corollary \ref{Corollary:grothendieck}. The argument is indeed similar to the arguments of \cite[Section 4]{grothendieck1966theoreme}; the primary new input is \cite[2.5]{de1998homomorphisms}. However, some complications arise due to the fact that reduced schemes are not necessarily geometrically reduced over imperfect ground fields.

First, we record the following lemma, which is uniqueness of analytic continuation.
\begin{lem}
\label{Lemma:formal_maps_agree}Let $X$ be an integral noetherian scheme and let $Y$ be a separated scheme. Let $Z\subsetneqq X$ be a non-empty closed subscheme. Let $u,v\colon X\rightarrow Y$ be two morphisms that agree when restricted to the formal scheme $X_{/Z}$. Then $u=v$.
\end{lem}

\begin{proof}
A morphism $X\rightarrow Y$ is uniquely determined by its restriction to any non-empty (dense) open $U\subset X$ because $X$ is reduced and $Y$ is separated. Therefore, to check the lemma we may restrict to any non-empty open subset $U\subset X$. Denote by $Z_{n}$ the $n$\textsuperscript{th} formal neighborhood of $Z$. Then the data of the map $X_{/Z}\rightarrow X$ is equivalent to the data of the maps $Z_{n}\rightarrow X$ \cite[Tag 0AI6]{stacks-project}. We wish to prove that the set of closed immersions $\{Z_{n}\rightarrow X\}_{n\geq0}$ are jointly schematically dense \cite[D\'efinition 11.10.2]{grothendieckegaIV_3}.

Pick an open affine $\Spec(R)\subset X$ that intersects $Z$ non-trivially and let $I$ be the ideal of $R$ corresponding to the scheme theoretical intersection of $Z$ with $\Spec(R)$. As $R$ is a noetherian domain, Krull's intersection theorem implies that the map of rings
\[
R\rightarrow{\displaystyle \lim_{n}}R/I^{n}
\]
is injective and criterion \cite[11.10.1(a)]{grothendieckegaIV_3} then implies the desired schematic density.
\end{proof}

The following definition will be of use to us.
\begin{defn}An integral domain $A$ is called a \emph{Krull domain} if the following three conditions hold.\begin{enumerate}
\item For each height one prime ideal $\mathfrak{p}$ of $A$, the localization $A_{(\mathfrak{p})}$ is a discrete valuation ring.
\item We have the following equality: $$\displaystyle A=\cap_{\mathfrak p}A_{(\mathfrak p)}$$
where $\mathfrak{p}$ runs through the height one prime ideals of $A$ and the intersection takes place in $\text{Frac}(A)$. 
\item Any non-zero $x\in A$ is contained in only finitely many height one prime ideals of $A$.
\end{enumerate}
An integral scheme $X$ is called an \emph{integral Krull scheme} if and only if for every open affine $\Spec(R)\subset X$, the ring $R$ is a Krull domain.
\end{defn}
A noetherian domain is Krull if and only if it is normal. The utility of the notion of a Krull domain is the following: while the integral closure of a noetherian domain need not be noetherian, it is always Krull by a theorem of Mori-Nagata \cite[Chapitre 0, 23.2.8 (iii)]{grothendieckegaIV_0}.

We now recall a well-known theorem of de Jong and Tate.
\begin{thm}
\label{Theorem:extending_BT_dJ}(de Jong-Tate) Let $X$ be an integral Krull scheme with generic point $\eta$. Let $G$ and $H$ be $p$-divisible groups over $X$. Then the natural map
\begin{equation}\label{equation:hom_BT_eta}
\Hom_{X}(G,H)\rightarrow\Hom_{\eta}(G_{\eta},H_{\eta})
\end{equation}
is an isomorphism. Moreover, if $X$ is noetherian and $f\in \Hom_X(G,H)$ is an isogeny when restricted to $\eta$, then $f$ is an isogeny.
\end{thm}
\begin{proof}
We first prove that Equation \ref{equation:hom_BT_eta} is an isomorphism. We immediately reduce to the case where $X\cong \Spec(R)$. As $R$ is a Krull domain,
\[
R=\cap R_{\mathfrak{p}}
\]
where $\mathfrak{p}$ runs over all of the height 1 primes. By the other defining characteristic of Krull domains, if $\mathfrak{p}$ is a height 1 prime, then $R_{\mathfrak{p}}$ is a discrete valuation ring. We are therefore reduced to proving the lemma over a discrete valuation ring, which results from \cite[Corollary 1.2]{de1998homomorphisms} when the generic fiber has characteristic $p$ and \cite[Theorem 4]{tate1967p} when the generic fiber has characteristic 0.

Now, if $X$ is noetherian and $f\colon G\rightarrow H$ is a homomorphism, the set of points where it is an isogeny is open by \cite[3.3.8]{chai2014complex} and closed by the discussion after \cite[3.3.10]{chai2014complex}. The result follows.

\end{proof}

\begin{defn}
Let $K/k$ be a field extension. We say that it is \emph{primary} if $k$ is separably closed in $K$ and \emph{regular} if $k$ is algebraically closed in $K$.
\end{defn}

The following proposition directly imitates Conrad's descent-theoretic proof of a theorem of Chow \cite[Theorem 3.19]{conrad2006chow}.
\begin{prop}
\label{Proposition:homomorphism_BT_primary_extension}Let $K/k$ be a primary extension of fields. Let $G$ and $H$ be BT groups on $k$. Then the natural injective map
\[
\Hom_{k}(G,H)\rightarrow\Hom_{K}(G_{K},H_{K})
\]
is an isomorphism.
\end{prop}

\begin{proof}
As $\Hom_{K}(G_{K},H_{K})$ is a finite free $\Zp$-module, we immediately reduce to the case that $K/k$ is finitely generated. Let $K'\cong K\otimes_{k}K$. Then we have a diagram
\[
k\rightarrow K\rightrightarrows K'
\]
where the double arrows refer to the maps $K\rightarrow K\otimes_{k}K$ given by $p_{1}\colon\lambda\mapsto\lambda\otimes1$ and $p_{2}\colon\lambda\mapsto1\otimes\lambda$ respectively. Let $F\in\Hom_{K}(G_{K},H_{K})$. By Grothendieck's
faithfully flat descent theory, to show that $F$ is in the image of $\text{Hom}_{k}(G,H)$, it is enough to prove that $p_{1}^{*}F=p_{2}^{*}F$ under the canonical identifications $p_{1}^{*}G\cong p_{2}^{*}G$ and $p_{1}^{*}H\cong p_{2}^{*}H$ \cite[Theorem 3.1]{conrad2006chow}. There is a distinguished \emph{diagonal point}: $\Delta:K'\twoheadrightarrow K$, and restricting to the diagonal point the two pullbacks agree $F=\Delta^{*}p_{1}^{*}F=\Delta^{*}p_{2}^{*}F$. If we prove that the
natural map
\[
\Hom_{K'}(G_{K'},H_{K'})\overset{\Delta^{*}}{\rightarrow}\Hom_{K}(G_{K},H_{K})
\]
is injective, we would be done. As $K/k$ is primary and finitely generated, it suffices to treat the following two cases.
\begin{casenv}
\item Suppose $K/k$ is a finite, purely inseparable extension of characteristic $p$. Then $K'$ is an Artin local ring. Moreover, the ideal $I=\text{ker}(\Delta)$ is nilpotent. As $K'$ is a ring which is killed by $p$ and $I$ is nilpotent, the desired injectivity follows from a lemma of Drinfeld
\cite[Lemma 1.1.3(2)]{katz1981serre} in his proof of the Serre-Tate theorem. Alternatively, it directly follows from \cite[1.4.4.3]{chai2014complex}.
\item Suppose $K/k$ is a finitely generated regular extension. Then $K'$ is a noetherian integral domain. The desired injectivity follows from localizing and again applying \cite[1.4.4.3]{chai2014complex}.
\end{casenv}
\end{proof}

We have the following useful corollary.
\begin{cor}\label{corollary:quasi-polarizations_extend}Let $X/k$ be a normal geometrically connected scheme of finite type. Let $G_0$ be a BT group over $k$ and set $G_{0,X}$ to be the pullback to $X$. Then the following pullback map is an isomorphism:
$$QPol(G_0)\rightarrow QPol(G_{0,X}).$$

\end{cor}
\begin{proof}
First of all, recall that a quasi-polarization is simply an anti-symmetric isogeny. As $X$ is normal, it follows that $X/k$ is geometrically irreducible. Set $K$ to the the function field $k(X)$ and $G_{0,K}$ to be the pullback of $G_0$ to $K$. As $X/k$ is geometrically irreducible and of finite type, $K/k$ is a finitely generated primary extension. By Proposition \ref{Proposition:homomorphism_BT_primary_extension}, it follows that $\text{Hom}_k(G_0,G_0^t)\rightarrow \text{Hom}_K(G_{0,K},G_{0,K}^t)$ is an isomorphism. Therefore the restriction map $QPol(G_0)\rightarrow QPol(G_{0,K})$ is an isomorphism.

There is a restriction map $QPol(G_{0,X})\rightarrow QPol(G_{0,K})$ and it suffices to prove that this map is an isomorphism. As $X$ is noetherian and normal, this follows from Theorem \ref{Theorem:extending_BT_dJ}.

\end{proof}

We now seek to prove a $p$-adic variant of \cite[Proposition 4.4]{grothendieck1966theoreme}: Lemma \ref{Lemma:descend_AV_via_BT}.  We first briefly recall the properties of the Chow $K/k$ trace for a primary extension of fields.
\begin{defn}Let $K/k$ be a finitely generated primary extension of fields and let $A/K$ be an abelian variety over $K$. The \emph{Chow $K/k$ trace}, denoted $\text{Tr}_{K/k}(A)$ is the final object in the category of pairs $(B,f)$ where $B/k$ is an abelian variety over $k$ equipped with a $K$-map $f\colon B_K\rightarrow A$.
\end{defn} 
We have the following properties of the Chow trace, which follow from \cite[Theorem 6.2, Theorem 6.4(2)]{conrad2006chow}: 
\begin{enumerate}
\item it exists and is functorial;
\item if $A$ is the base change of an abelian variety $B/k$, then $\text{Tr}_{K/k}(A)\cong B$;
\item the kernel of the universal map $f_{\text{univ}}\colon (\text{Tr}_{K/k}(A))_K\rightarrow A$ is finite.
\end{enumerate}

\begin{prop}
\label{Proposition:descent_AV_via_BT_purely_inseparable}Let $K/k$ be a finite, purely inseparable field extension. Let $A/K$ be an abelian variety such that $A[p^{\infty}]$ is the pullback of a BT group $G_{0}$ over $k$. Then $A$ is pulled back from an abelian variety $B_0$ over $k$.
\end{prop}

\begin{proof}
Set $B_0:=\text{Tr}_{K/k}(A)$. By definition, there is a natural map $B_{0,K}\rightarrow A$ and by \cite[Theorem 6.6]{conrad2006chow} this map is an isogeny (as $(K\otimes_{k}K)_{red}\cong K$, see the Terminology and Notations of \emph{ibid.}) with connected Cartier dual; in particular the kernel $I$ has order a power of $p$. The induced isogeny of BT groups $B_{0,K}[p^{\infty}]\rightarrow G_{0,K}\cong A[p^{\infty}]$ also has kernel exactly $I$. Applying Proposition \ref{Proposition:homomorphism_BT_primary_extension}, we see that this isogeny descends to an isogeny $B_{0,K}[p^{\infty}]\rightarrow G_{0}$ over $k$. Call the kernel of this isogeny $H$; then $H$ is a finite $k$-group scheme with $H_{K}\cong I$. Finally, $B_0/H$ is an abelian variety over $k$ and the induced map
\[
(B_0/H)_{K}\rightarrow A
\]
is an isomorphism. (Indeed, as a byproduct of the proof we deduce that $H$ is trivial.)
\end{proof}

With Proposition \ref{Proposition:descent_AV_via_BT_purely_inseparable} in hand, we are ready to prove the following important lemma.
\begin{lem}
\label{Lemma:descend_AV_via_BT}Let $X/k$ be a normal, geometrically connected scheme of finite type over a field of characteristic $p$ and let
$\pi\colon A\rightarrow X$ be an abelian scheme. Then the following are equivalent 
\begin{enumerate}
\item The BT group $A[p^{\infty}]$ is isomorphic to the pullback of a BT group $G_{0}$ over $k$.
\item There abelian scheme $A$ is the pullback of an abelian scheme $B_0$ over $k$. 
\end{enumerate}
\end{lem}

\begin{proof}
That (2) implies (1) is evident, so we prove the reverse implication. We will explain a series of reductions. First of all, as $X/k$ is normal and geometrically connected, it is geometrically integral. We claim that that it is equivalent to prove the lemma over the generic point $\eta=\Spec(k(X))$. Indeed, suppose there exists an abelian variety $B_0/k$ such that $B_{0,\eta}\cong A_{\eta}$. Then $B_{0,X}\cong A$ by \cite[Ch. I., Proposition 2.7]{faltings2013degeneration}.

The hypotheses on $X/k$ imply that $k(X)/k$ is a finitely generated primary extension. Note that if (2) holds, then $B_0\cong \text{Tr}_{k(X)/k}A_{\eta}$. Set $B_{0}=Tr_{k(X)/k}A_{\eta}$; then $B_{0}$ is an abelian variety over $k$ equipped with a homomorphism $(B_{0})_{\eta}\rightarrow A_{\eta}$ with finite kernel. We will use functorial properties of the Chow trace to reduce to the case when $k$ is algebraically closed.

Let $k'$ be the algebraic closure of $k$ in $k(X)$. Then $k'/k$ is a finite, purely inseparable extension. If we prove the theorem for the geometrically integral scheme $X/k'$, then by Proposition \ref{Proposition:descent_AV_via_BT_purely_inseparable} we will have proven the theorem for $X/k$. Therefore we reduce to the case that $X/k$ is geometrically integral.

We now explain why we may reduce to the case $k=\bar{k}$. As $X/k$ is now supposed to be geometrically integral, $k(X)/k$ is a regular field extension and \cite[Theorem 5.4]{conrad2006chow} implies that the construction of $k(X)/k$ trace commutes with \emph{arbitrary} field extensions $E/k$. On the other hand, the $k(X)$-morphism $B_{0,\eta}\rightarrow A_{\eta}$ is an isomorphism if and only if the its base-change to $\bar{k}(X)$ is an isomorphism.

Next we explain why we may reduce to the case that $A[l]\rightarrow X$ is isomorphic to a trivial \'etale cover (maintaining the assumption that $k=\bar{k}$). Let $Y\rightarrow X$ be the finite, connected Galois cover that trivializes $A[l]$. Suppose there exists an abelian variety $C_{0}/k$ such that $A_{Y}\cong C_{0,Y}$ as abelian schemes over $Y$. Let $I:=\textbf{Isom}(C_{0,X},A)$ be the $X$-scheme of isomorphisms between $C_{0,X}$ and $A$. Our assumptions that $I(Y)$ is non-empty and that there is an isomorphism $A[p^{\infty}]\cong G_{0,X};$ our goal is to prove that $I(X)$ is non-empty.

Pick isomorphisms $C_{0}[p^{\infty}]\cong G_{0}$ and $A[p^{\infty}]\cong G_{0,X}$
once and for all. Then we have a commutative square of sets, where
the horizontal arrows are injective because the $p$-divisible group
of an abelian scheme is relatively schematically dense:
\[
\xymatrix{I(X)\ar[d]\ar[r] & \text{Aut}_{X}(G_{0,X})\ar[d]\\
I(Y)\ar[r] & \text{Aut}_{Y}(G_{0,Y}).
}
\]

By Proposition \ref{Proposition:homomorphism_BT_primary_extension}, the maps $\Aut_{k}(G_{0})\rightarrow \Aut_{X}(G_{0,X})\rightarrow \Aut_{Y}(G_{0,Y})$ are all isomorphisms: indeed, both $k(Y)/k$ and $k(X)/k$ are regular extensions. Let $H=Gal(Y/X)$. Then $H$ acts naturally on $G_{0,Y}$ and hence acts on $I(Y)\hookrightarrow Aut_{Y}(G_{0,Y})$ via conjugation. On the other hand, the action of $H$ on $Aut_{Y}(G_{0,Y})$ is trivial because the fixed points are $Aut_{X}(G_{0,X})$; therefore the action of $H$ on $I(Y)$ is trivial. Our running assumption was that we had an element of $I(Y)$. Then Galois descent implies that $I(X)$ is non-empty, as desired.

We can find a polarization by \cite[Remark 1.10(a)]{faltings2013degeneration}, so assume we have one of degree $d$. By the above reductions, we may assume that $k=\bar{k}$ and that there exists an $l\geq3$ such that $A[l]\rightarrow X$ is isomorphic to a trivial \'etale cover. Let $x\in X(k)$ be a $k$-point and let $C_{0}:=A_{x}$. Denote by $C:=(C_{0})_X$ the constant abelian scheme with fiber $C_{0}$ and polarization induced from that of $A_x$. By assumption the BT groups $A[p^{\infty}]$ and $C[p^{\infty}]$ are isomorphic. They are in fact isomorphic as \emph{quasi-polarized} BT groups by Corollary \ref{corollary:quasi-polarizations_extend}.

There are two induced moduli maps from $X$ to the separated (fine)
moduli scheme $\mathscr{A}_{g,d,l}$. Consider the restrictions of these moduli maps to the formal scheme $X_{/x}$: Serre-Tate theory \cite[Theorem 1.2.1]{katz1981serre} ensures that the two formal moduli maps from $X_{/x}$ to $\mathscr{A}_{g,d,l}$ agree. Then by Lemma \ref{Lemma:formal_maps_agree} the two moduli maps agree on all of $X$, whence the conclusion.
\end{proof}

We will use the following easy refinement.
\begin{cor}
\label{Corollary:constant_BT_level_polarization}Let $X/k$ be a normal, geometrically connected scheme of finite type over a field of characteristic $p$. Let $p\nmid N$ and $d$ be positive integers. Let $A\rightarrow X$ be an abelian scheme equipped with a polarization $\lambda$ of degree $d$ and a full level $N$ structure $\Gamma$. Suppose $A[p^{\infty}]$ is the pullback of a BT group $G_0$ over $k$. Then $(A,\lambda,\Gamma)$ is the pullback of a polarized abelian variety with full level-$N$ structure $(A_0,\lambda_0,\Gamma_0)$ over $k$.
\end{cor}

\begin{proof}
By Proposition \ref{Lemma:descend_AV_via_BT}, there exists an abelian variety $A_0$ over $k$ whose pullback to $X$ is isomorphic to $A$. Let $K:=k(X)$ and note again that $K/k$ is a primary field extension. The group scheme $A_0[N]$ is \'etale, and a non-trivial \'etale group scheme cannot become trivial after a primary field extension; therefore the full level-$N$ structure descends. We need to show the polarization descends to a polarization on $A_0$. A theorem of Chow \cite[Theorem 3.19]{conrad2006chow} implies that any homomorphism $A_K\rightarrow A_K^t$ is the pullback of a unique homomorphism $\lambda_0\colon A_0\rightarrow A_0^t$. There exists an ample line bundle $\calL$ on $A_{\bar K}$ such that $\lambda$ comes from $\calL$ in the usual way. On the other hand, the N\'eron-Severi group of $A_{0,\bar k}$ and the N\'eron-Severi group of $A_{\bar K}$ are isomorphic; hence one can find such a line bundle $\calL$ (in the same N\'eron-Severi class) that is defined over $\bar k$. This implies that $\lambda_0$, as a symmetric isogeny, is in fact a polarization as in \cite[Ch. I, Definition 1.6]{faltings2013degeneration}.
\end{proof}

Before we prove the main result of this section, we need one more auxiliary result which provides an easy-to-check criterion that characterizes when a map from an open dense subset $U$ of a noetherian integral normal scheme $X$ extends to all of $X$. The proposition itself is an easy corollary of Zariski's main theorem.

\begin{prop}\label{Proposition:extend_map_on_open}Let $X$ be a noetherian integral normal scheme and let $\mathcal{M}$ be a reduced and separated scheme. Let $U\subset X$ be a non-empty open subset and let $f_U\colon U\rightarrow \mathcal{M}$ be a morphism. Let $X'\subset X\times \mathcal{M}$ be the Zariski closure of the graph $\Gamma_U$ of $f_U$, equipped with induced reduced structure, and suppose $X'\rightarrow X$ is proper. For each $x\in X$, let $X'_x$ denote the fiber of $X'$ over $x$. Then the following are equivalent.
\begin{enumerate}
\item There exists a (necessarily unique) extension $f\colon X\rightarrow \mathcal{M}$ of $f$.
\item For each $x\in X\backslash U$, there exists a reduced scheme $\tilde{Y}$ together with a finite, dominant morphism $\tilde{Y}\rightarrow (X_x')_{\text{red}}$ such that for each irreducible component $\tilde{Y}_j$ of $\tilde{Y}$, the composed map $$\tilde{Y}_j\rightarrow (X_x')_{\text{red}}\hookrightarrow X'\rightarrow \mathcal{M}$$
has image supported at a single point.
\end{enumerate}
\end{prop}
\begin{proof}
As $\pi\colon X'\rightarrow X$ is proper and birational, it follows that $X'$ is a noetherian, integral scheme. Then, as $X$ is normal, it follows from Zariski's main theorem that for every $x\in X$, the fiber $X_x$ is geometrically connected. For $x\in X\backslash U$, let $k:=k(x)$. Then $(X'_x)_{\text{red}}$ is proper over $\Spec(k)$. As $\mathcal M$ is separated, the image of $(X'_x)_{\text{red}}\hookrightarrow \mathcal M$ is closed.

Zariski's main theorem implies that the following three statements are equivalent.
\begin{itemize}
\item The map $\pi$ is an isomorphism.
\item The image of $(X'_x)_{\text{red}}\rightarrow \mathcal{M}$ is a single point for all $x\in X\backslash U$.
\item The (automatically closed) image of  $(X'_x)_{\text{red}}\rightarrow \mathcal{M}$ is zero-dimensional for all $x\in X\backslash U$.
\end{itemize} 
Now, (1) clearly implies (2). To prove the converse, note that each $\tilde{Y}_i$ is proper over $\Spec(k)$; therefore the image of $\tilde{Y}_i\rightarrow (X'_x)\hookrightarrow \mathcal{M}$ is closed. Therefore, if (2) holds, then $\pi$ is an isomorphism as desired.
\end{proof}

\begin{cor}
\label{Corollary:grothendieck}Let $X$ be a locally noetherian normal scheme and $U\subset X$ be an open dense subset whose complement has characteristic $p$. Let $A_{U}\rightarrow U$ be an abelian scheme. Then $A_{U}$ extends to an abelian scheme over $X$ if and only if $A_{U}[p^{\infty}]$ extends to a $p$-divisible group over $X$.
\end{cor}

\begin{proof}
When $X$ is the spectrum of a discrete valuation ring, this follows from Grothendieck's generalization of N\'eron-Ogg-Shafarevitch criterion in mixed characteristic \cite[Expos\'e IX, Theorem 5.10]{grothendieck2006groupes}
and work of de Jong \cite[2.5]{de1998homomorphisms} in equal characteristic $p$. The rest of the argument closely follows \cite[pages 73-76]{grothendieck1966theoreme}.

We review the reductions on \cite[page 73]{grothendieck1966theoreme}. First of all, if such an extension exists, it is unique up to unique isomorphism by \cite[1.2]{grothendieck1966theoreme}. Therefore, to prove the existence of such an extension, it is harmless to replace $X$ by a finite \'etale cover by descent theory. As $X$ is locally noetherian and normal, every connected component of $X$ is irreducible \cite[Tag 0357]{stacks-project}. Therefore, we may suppose that $X$ is integral, $A_{U}\rightarrow U$ has a polarization of degree $d$, and the $l$-torsion is of $A_U$ isomorphic to the trivial \'etale cover of $U$ for some prime $l\geq 3$. In particular, we have a map $f\colon U\rightarrow\mathscr{A}_{g,d,l}$, where the latter fine moduli space is a separated scheme of finite type over $\mathbb{Z}[1/l]$.

Let $X'\subset X\times \mathscr{A}_{g,d,l}$ denote the closure of the graph of $f$, equipped with reduced, induced scheme structure. Then $X'$ is integral and the natural map $X'\rightarrow X$ is of finite type. The valuative criterion of properness together with the case of discrete valuation rings discussed above imply that the morphism $\pi\colon X'\rightarrow X$ is proper and birational. Our goal is to show that $U\rightarrow \mathscr{A}_{g,d,l}$ extends to a morphism $X\rightarrow \mathscr{A}_{g,d,l}$. We will use Proposition \ref{Proposition:extend_map_on_open}; the key is to find an appropriate $\tilde{Y}$, the process of which will pass through a scheme that is not \`a priori noetherian.

Let $X''\rightarrow X$ denote the normalization of $X'$. It is \emph{not \`a priori true} that $X''\rightarrow X$ is finite (and $X''$ need not be noetherian). Nonetheless, as $X'$ is integral and noetherian, a theorem of Mori-Nagata implies that the fibers of $X''\rightarrow X'$ have \emph{finite reduction} in the following sense: if $F$ is the fiber over a point $x'$ of $X'$, then $F_{\text{red}}\rightarrow x'$ is finite \cite[Chapitre 0, 23.2.6, 23.2.7]{grothendieckegaIV_0}. Moreover, $X''$ is automatically an integral Krull scheme \cite[Chapitre 0, 23.2.8(iii)]{grothendieckegaIV_0}. Let $B_{X'}$ and $B_{X''}$ be the abelian schemes on $X'$ and $X''$ respectively, induced from the map $X''\rightarrow X'\rightarrow\mathscr{A}_{g,d,l}$.

For every point $x\in X\backslash U$,  set $Z:=(X'_x)_{\text{red}}$, set $Y:=(X''_x)_{\text{red}}$, and set $k:=k(x)$; by assumption $\text{char}(k)=p$ and by Zariski's main theorem, $Z/k$ is geometrically connected of finite type. By the above result of Mori-Nagata, $Y/k$ is also of finite type.
Let $G$ be the BT group that exists by assumption on $X$ and let $G_{X''}$ be the pullback of $G$ to $X''$. Then $B_{X''}[p^{\infty}]$ and $G_{X''}$ are isomorphic on $U\subset X''$. As $X''$ is an integral Krull scheme, Theorem \ref{Theorem:extending_BT_dJ} implies that $B_{X''}[p^{\infty}]\cong G_{X''}$ \emph{on all of $X''$}. Therefore for every point $x$ of $X$, the abelian scheme $B_{Y}$ has constant BT group over $Y$ with respect to $k$. (We emphasize that we do not yet know that $B_{X'}[p^{\infty}]$ and $G_{X'}$ are isomorphic because the noetherian scheme $X'$ is not necessarily normal and hence not necessarily Krull; hence we do not yet know that $B_Z$ has constant BT group over $Z$ with respect to $k$.)

Write $Z=\cup Z_i$ to be the decomposition of $Z$ into irreducible components. For each $i$, set $z_i$ to be the generic point of $Z_i$. Let $y_i$ be a point of $Y$ that maps to $z_i$; note that $k(y_i)$ is a finite field extension of $k(z_i)$. Let $Y_i$ be the Zariski closure of $y_i$ in $Y$ with reduced induced structure scheme structure ($Y_i$ is not necessarily normal). Let $\tilde{Y}_i$ be the normalization of $Y_i$. As $Y_i\rightarrow Z$ is of finite type and $Z/k$ is of finite type over $k$, it follows from a theorem of E. Noether that $\tilde{Y}_i\rightarrow Z_i$ is finite  \cite[Corollary 13.13]{eisenbud2013commutative}. Set $\tilde{Y}:=\sqcup_i \tilde{Y}_i$; then $\tilde{Y}\rightarrow Z$ is finite and dominant. 

Consider the map $\tilde{Y}\rightarrow Z\rightarrow \mathscr{A}_{g,d,l}$; for each $i$ the induced abelian scheme $B_{\tilde{Y}_i}\rightarrow \tilde{Y}_i$ has constant BT group on $\tilde{Y}_i$ (with respect to $k$) because the same is true over $Y$. For each $i$, let $\tilde{k}_i$ be the algebraic closure of $k$ in $k(y_i)$. (The extension $\tilde{k}_i/k$ is a finite extension.) Then $\tilde{Y}_i/\tilde{k}_i$ is a normal, geometrically connected scheme of finite type. We claim that the composite map $$\tilde{Y}_i\rightarrow Z_i\rightarrow X'\rightarrow \mathscr{A}_{g,d,l}$$ factors through the map $\tilde{Y}_i\rightarrow \Spec(\tilde{k}_i)$ and in particular has closed image supported at a single point; indeed, this follows immediately from Corollary \ref{Corollary:constant_BT_level_polarization} and the fact that $\mathscr{A}_{g,d,l}$ is a fine moduli space.
%

For each $x\in X\backslash U$, the scheme $\tilde{Y}\rightarrow Z$ satisfies the conditions of Proposition \ref{Proposition:extend_map_on_open}; by applying the Proposition, we conclude.

\end{proof}

\section{Lefschetz Theorems I}\label{Section:Lefschetz I}
In this section, we collect together several Lefschetz theorems having to do with morphisms of abelian varieties and $p$-divisible groups.
We first state the simplest case of a theorem of a Lefschetz theorem of Abe-Esnault, which will be useful for us to prove Lemma \ref{Lemma:lefschetz_hom_BT}, a Lefschetz theorem for homomorphisms of BT groups.
\begin{thm}
\label{Theorem:AE_Lefschetz}(Abe-Esnault) Let $X/k$ be a smooth projective variety over a perfect field $k$ of characteristic $p$ with $\dim X\geq2$ and let $C\subset X$ be a smooth projective curve that is the complete intersection of smooth ample divisors. Then for any $\mathcal{E}\in\fisoc{X}_{\Qpbar}$, the following restriction map is an isomorphism
\[
H^{0}(X,\mathcal{E})\rightarrow H^{0}(C,\mathcal{E})
\]
\end{thm}

\begin{proof}
Note that $\mathcal{E}$ is automatically overconvergent as $X$ is projective. This then follows immediately from the arguments of \cite[Corollary 2.4]{abe2016lefschetz}.
We make one further comment. While the statements \cite[Corollary 2.4, Proposition 2.2]{abe2016lefschetz}
are only stated for $\mathbb{F}$, they only require $k$ be perfect as stated in \cite[Remark 2.5]{abe2016lefschetz}. Indeed, Abe and Esnault pointed out via email that the part of \cite{abe2013theory} they cite for these arguments is Section 1, which only requires $k$ be perfect.

\end{proof}
\begin{lem}
\label{Lemma:lefschetz_hom_BT}Let $X/k$ be a smooth projective variety over a field $k$ of characteristic $p$ with $\dim X\geq2$ and let $U\subset S$ be a Zariski open subset. Let $\calG_U$ and $\calH_U$ be BT groups on $U$. Let $C\subset U$ be a smooth projective curve that is the complete intersection of smooth ample divisors of $X$ and denote by $\calG_{C}$ (resp. $\calH_C$) the restriction of $\calG_U$ (resp. $\calH_U$) to $C$. Then the following restriction map 
\begin{equation}\label{Eqn:restriction_hom_BT}
\Hom(\calG_{U},\calH_{U})\rightarrow\Hom_{C}(\calG{}_{C},\calH_{C})
\end{equation}
is injective with cokernel killed by a power of $p$.
\end{lem}

\begin{proof}
First of all, note that $U\subset X$ has complementary codimension at least two. Indeed, $U$ contains $C$, which is a projective curve that is the smooth complete intersection of ample divisors of $X$; therefore $C$ intersects every irreducible divisor of $X$ non-trivially. 

We now reduce to the case that $k$ is perfect. Let $k'/k$ be a
field extension and let $U':=U\times_{k}k'$. By Theorem \ref{Theorem:extending_BT_dJ},
the natural map 
\[
\Hom_{U'}(\calG_{U'},\calH_{U'})\rightarrow\Hom_{k'(U)}(\calG_{k'(U)},\calH_{k'(U)})
\]
is an isomorphism and similarly for $D$. Let $l/k$ be the perfect closure; it is in particular a primary extension. Therefore $l(U)/k(U)$ and $l(D)/k(D)$ are also primary extensions and hence Proposition \ref{Proposition:homomorphism_BT_primary_extension} implies both sides of the equation are unchanged when we replace $k$ by $l$. 

 Consider the following diagram: 
\[
\textbf{F-Isoc}^{\dagger}(X)\stackrel{\sim}{\rightarrow}\textbf{F-Isoc}(X)\rightarrow\textbf{F-Isoc}(U)
\]
The first arrow is an equivalence of categories because $X$ is proper and the second arrow an equivalence of categories by work of Kedlaya and Shiho \cite[Theorem 5.1(c)]{kedlaya2016notes} because the complement of $U$ has codimension at least 2. Hence $\mathbb{D}(\calG_{U})\otimes\QQ$ and $\mathbb{D}(\calH_{U})\otimes\QQ$ have canonical extensions to $X$, which we denote by $\calE$ and $\calF$ respectively. 

Both the left and the right hand side of Equation \ref{Eqn:restriction_hom_BT} are torsion-free $\ZZ_{p}$-modules. We have the following commutative diagram:
\[
\xymatrix{\Hom_{U}(\calG_{U},\calH_{U})\ar[rr]^{\sim}\ar[d]_{\text{res}} &  & \Hom_{U}(\mathbb{D}(\calH_{U}),\mathbb{D}(\calG_{U}))\ar[d]^{\otimes\QQ}\\
\Hom_{C}(\calG_{C},\calH_{C})\ar[d]_{\otimes\QQ} &  & \Hom_{U}(\mathbb{D}(\calH_{U})\otimes\QQ,\mathbb{D}(\mathcal{G}_{U})\otimes\QQ)\\
\Hom_{C}(\mathbb{D}(\calH_{C})\otimes\QQ,\mathbb{D}(\calG_{C})\otimes\QQ) &  & \Hom_{X}(\calF,\calE)\ar[ll]\ar[u]_{\wr}
}
\]
The lower right hand vertical arrow is an isomorphism,
again by the above equivalence of categories. Moreover we have the
isomorphisms
\[
\mathbb{D}(\calG_C)\otimes\QQ\cong\calE|_{C},\text{ and }\mathbb{D}(\calH_C)\otimes\QQ\cong\calF_C
\]
Our goal is to prove that $\text{res}$, an application of the ``restrict to $C$'' functor, is injective with cokernel killed by a power of
$p$. As $\Hom_{C}(\calG_C,\calH_C)$ is a finite free $\mathbb{Z}_{p}$-module, it is equivalent to prove
that the induced map
\[
\Hom_U(\calG_{U},\calH_{U})\otimes\QQ\rightarrow\Hom_C(\mathbb{D}(\calH_C)\otimes\QQ,\mathbb{D}(\calG_C)\otimes\QQ)
\]
is an isomorphism of $\Qp$-vector spaces. By chasing the diagram, this is equivalent to the bottom horizontal arrow being an
isomorphism. This arrow is an isomorphism due to Theorem \ref{Theorem:AE_Lefschetz} (which, as stated, requires $k$ to be perfect).
\end{proof}
In a similar vein, we write down a Lefschetz theorem for homomorphisms of abelian schemes.
\begin{thm}
\label{Theorem:Lefschetz_AV}Let $X/k$ be a smooth projective variety
over a field $k$ with $\dim X\geq2$ and let $U\subset X$ be a Zariski open subset. Let $A_{U}\rightarrow U$ and $B_{U}\rightarrow U$ be abelian schemes over $U$. Let $C\subset U$
be a smooth projective curve that is the complete intersection of smooth, ample divisors of $X$, and denote by $A_{C}$ (resp. $B_{C})$ the restriction of $A_{U}$ (resp. $B_{U}$) to $C$. Then the natural restriction map

\begin{equation}
\Hom_{U}(A_{U},B_{U})\rightarrow\Hom_{C}(A_{C},B_{C})\label{eq:restriction}
\end{equation}
is an isomorphism when tensored with $\QQ$. If the cokernel is non-zero, then $\text{char}(k)=p$ and the cokernel is killed by a power of $p$. 
\end{thm}

\begin{proof}
We first remark that the above map is well-known to be injective and does not require any positivity property of $D$ (one may immediately reduce to the case of a discrete valuation ring).

We now reduce to the case when $k$ is a finitely generated field. Both sides of Equation \ref{eq:restriction} are finitely generated $\ZZ$-modules; hence we my replace $k$ by a subfield finitely generated over the prime field over which everything in Equation \ref{eq:restriction} is defined \emph{without changing either the LHS or the RHS}. As $C\subset U$ and $C$ is the complete intersection of smooth ample divisors of $X$, it follows that the complementary codimension of $U$ in $X$ is at least 2. Note further that by Grothendieck's Lefschetz theorem \cite[Expos\'e X, \S 2, Corollaire 2.6]{grothendieck2005cohomologie}, Zariski-Nagata purity, and induction, the map $\pi_1(C)\rightarrow \pi_1(U)$ is surjective.

By \cite[Ch. I, Prop. 2.7]{faltings2013degeneration}, the natural
map
\begin{equation}
\Hom(A_{U},B_{U})\rightarrow\Hom_{k(U)}(A_{k(U)},B_{k(U)})\label{eq:extend_av}
\end{equation}
is an isomorphism, and similarly for $C$. Now, as both sides of Equation \ref{eq:restriction} are finite free $\mathbb{Z}$-modules, it suffices to examine its completion at all finite places. There are two cases.
\begin{casenv}
\item $\text{char}(k)=0$. Then ``Tate's isogeny theorem'' is true for
$U$ and $C$; that is, the natural map
\[
\Hom_{U}(A_{U},B_{U})\otimes\ZZ_l\rightarrow\Hom_{\pi_{1}(U)}(T_{l}(A_{U}),T_{l}(B_{U}))
\]
is an isomorphism for every $l$, and similarly for $C$; this is a combination of work of Faltings \cite[Theorem 1, Page 211]{faltings1984rational} and Equation \ref{eq:extend_av}. On the other hand, as noted above, the restriction map $\pi_{1}(C)\rightarrow\pi_{1}(U)$
is surjective; therefore the natural map
\[
\Hom_{\pi_{1}(U)}(T_{l}(A_{U}),T_{l}(B_{U}))\rightarrow\Hom_{\pi_{1}(C)}(T_{l}(A_{C}),T_{l}(B_{C}))
\]
is an isomorphism. Moreover, this is true for all $l$. Therefore Equation \ref{eq:restriction} is an isomorphism.
\item $\text{char}(k)=p$. Then ``Tate's isogeny theorem'', in this case
a theorem of Tate-Zarhin-Mori \cite[ Ch. XII, Th\'eor\`eme 2.5(i), p. 244]{moret1985pinceaux} together with \ref{eq:extend_av}, implies that the argument of part (1) works as long as $l\neq p$. Therefore Equation \ref{eq:restriction} is rationally an isomorphism and the cokernel is killed by a power of $p$.
\end{casenv}
\end{proof}
\begin{rem}
Daniel Litt indicated to us that there is a more elementary proof of Theorem \ref{Theorem:Lefschetz_AV} (i.e. not using Faltings' resolution of the Tate conjecture for divisors on abelian varieties) when $\text{char}(F)=0$ along the lines of his thesis. He also pointed out that, when $F\cong\CC$, the result follows from the theorem
of the fixed part.
\end{rem}

\begin{example}
\label{Example:Litt}Litt constructed an example when Equation \ref{eq:restriction} is not an isomorphism in characteristic
$p$. Let $E/\FF$ be a supersingular elliptic curve, so $\alpha_{p}\subset E$
is a subgroup scheme. Then $E\times E\times E$ contains $\alpha_{p}^{3}$
as a subgroup scheme. There is a injective homomorphism
\[
(\alpha_{p})_{\mathbb{P}^{2}}\hookrightarrow(\alpha_{p}^{3})_{\mathbb{P}^{2}}
\]
given as follows: if $[x:y:z]$ are the coordinates on $\mathbb{P}^{2}$
and $(\alpha,\beta,\gamma)$ are linear coordinates on $\alpha_{p}^{3}\subset\mathbb{A}^{3}$,
then the image of $\alpha_p$ is cut out by the equations:
\[
[x:y:z]=[\alpha:\beta:\gamma]
\]
Let $A\rightarrow\PP^{2}$ be the quotient of the constant $E\times E\times E$ family over $\PP^{2}$ by this varying
family of $\alpha_{p}$. This family admits a principal polarization and the induced image
\begin{equation}\label{equation:p2_supersingular}
\PP^2\rightarrow\calA_{3,1}\otimes\FF
\end{equation}
is infinite, see the text after \cite[page 59, 9.4.16]{li1998moduli}. Suppose there exists a line $H\subset\PP^2$, an automorphism $\phi$ of $\PP^2$ that fixes $H$ pointwise, and a point $p\in\PP^2\backslash H$ with the fibers $A_{p}$ and $A_{\phi^{-1}(p)}$ non-isomorphic as unpolarized abelian threefolds. Let $B:=\phi^{*}A$, a new family of abelian threefolds over $\PP^2$. Then $A|_{H}\cong B|_{H}$
but $A\ncong B$ as abelian schemes over $\PP^2$.

We now explain why we can always find such a triple $(H,\phi,p)$. It suffices to find two points $p,q\in\PP^{2}$ such that $A_{p}\ncong A_{q}$; we can then of course find an automorphism that sends $p$ to $q$ and fixes a line. If $B/k$ is an abelian variety over a field, then there are only finitely many isomorphism classes of pairs $(B,\lambda)$, where $\lambda\colon B\rightarrow B^t$ is a principal polarization \cite[Theorem 1.1]{narasimhan1981polarisations}. (As usual, an isomorphism $\lambda\colon B\rightarrow B^t$ is called a principal polarization if there exists an ample line bundle $L$ on $B_{\bar k}$ such that the map $\lambda\colon B_{\bar k}\rightarrow B_{\bar k}^t:=Pic^0(B_{\bar k})$ is the same as the map induced by $$b\mapsto m^*_bL\otimes L^{-1}.$$ This map only depends on the N\'eron-Severi class of $L$.) Moreover, the image of the map in Equation \ref{equation:p2_supersingular} is infinite. Therefore we can find $p,q\in\PP^2(\FF)$ such that $A_{p}$ is not isomorphic to $A_{q}$ as \emph{unpolarized} abelian varieties.  
\end{example}

The primary application of \ref{Lemma:lefschetz_hom_BT} is to show
that certain quasi-polarizations lift.
\begin{cor}
\label{Corollary:quasi_polarizations_lift}Let $X/k$ be a smooth projective variety over a field $k$ of characteristic $p$ with $\dim X\geq2$, let $U\subset X$ be a Zariski open subset, and let $C\subset U$ be a smooth complete curve that is the complete intersection of smooth ample divisors of $X$. Let $\calG_{U}$ be a BT group on $U$. Then the natural map
\[
QPol(\calG_U)\rightarrow QPol(\calG_C)
\]
is injective and moreover for any $\phi_{C}\in QPol(\calG_C)$, there exists $n\geq0$ such that $p^{n}\phi_{C}$ is the in the image of the above map.
\end{cor}

\begin{proof}
We have a commutative square
\[
\xymatrix{QPol(\calG_{U})\ar[r]\ar[d] & \Hom_{U}(\calG_{U},\calG_{U}^{t})\ar[d]\\
QPol(\calG_{C})\ar[r] & \Hom_{C}(\calG_C,\calG_C^{t})
}
\]
where the vertical arrow on the right is injective with torsion cokernel by Lemma \ref{Lemma:lefschetz_hom_BT}. Now, $QPol(\calG_C)$ are those isogenies in $\Hom_{C}(\calG_C,\calG_{C}^{t})$ that are anti-symmetric, i.e. those $\phi_{C}$ such that $\phi_{C}^{t}=-\phi_{C}$. Given $\phi_{C}\in QPol(\calG_{C})$, we know there exists an $n\geq0$ such that $p^{n}\phi_{C}$ is the image a unique $\psi\in\Hom_{U}(\calG_U,\calG_{U}^{t})$; to prove the corollary we must check that $\psi$ is an isogeny and anti-symmetric on all of $U$.

First of all, we remark that the transpose map commutes with the vertical right hand arrow, namely the restriction to $C$. As the vertical right hand arrow is injective, the skew-symmetry of $p^{n}\phi_{C}$ implies that $\psi$ is automatically skew-symmetric. The set of points where $\psi$ is an isogeny is open by \cite[3.3.8]{chai2014complex} and closed by the discussion after \cite[3.3.10]{chai2014complex}. Therefore $\psi$ is a quasi-polarization, as desired.
\end{proof}
\begin{cor}
\label{Corollary:formal_abelian_scheme}Let $X/k$ be a smooth projective variety over a field $k$ of characteristic $p$, let $C\subset X$ be a smooth curve that is the complete intersection of smooth, ample divisors of $X$, and let
$U\supset C$ be a Zariski open set of $X$. Let $A_{C}\rightarrow C$
be an abelian scheme and let $\calG_{U}$ be a BT group on $U$ such that $\calG_{C}\cong A_{C}[p^{\infty}]$ as BT groups on
$C$. Then there exists a polarizable formal abelian scheme $\hat{A}$ over the formal scheme $X_{/C}$ such that $\hat{A}[p^{\infty}]\cong\calG|_{X_{/D}}$.
The formal abelian scheme $\hat{A}$ is (uniquely) algebraizable.
\end{cor}

\begin{proof}
It follows from \cite[page 6]{faltings2013degeneration} that $A_{C}\rightarrow C$ is globally projective. Let $\phi_{C}$ be a polarization of $A_{C}$; abusing notation, we also let $\phi_{C}$ denote the
induced quasi-polarization on the BT group $\calG_{C}$.  By Corollary \ref{Corollary:quasi_polarizations_lift}, there exists an $n\geq0$ such that the quasi-polarization $p^{n}\phi_{C}$
lifts to a quasi-polarization $\psi$ of $\calG_{U}$. Applying the Serre-Tate theorem \cite[Theorem 1.2.1]{katz1981serre}, we therefore get a formal abelian scheme $\hat{A}$ on $X_{/D}$; the fact that the quasi-polarization lifts implies that $\hat{A}$ is polarizable and hence is (uniquely) algebraizable by Grothendieck's algebraization theorem \cite[TAG 089a]{stacks-project}.
\end{proof}

\section{Lefschetz Theorems II}\label{Section:Lefschetz 2}

Good references for this section are \cite{grothendieck2005cohomologie, hartshorneample}.
\begin{defn}
\cite[Expos\'e X, Section 2]{grothendieck2005cohomologie} We say a pair of a noetherian scheme $X$ and a non-empty closed subscheme $Z$ satisfies the property $Lef(X,Z)$ if for any vector bundle $E$ on a Zariski neighborhood $U\supset Z$, the following restriction map is an isomorphism:
\[
H^{0}(U,E)\cong H^{0}(X_{/Z},E|_{X_{/Z}})
\]
\end{defn}
\begin{defn}
(\cite[p.64]{hironakaformal} or \cite[Ch V, \S 1, p. 190]{hartshorneample}) Let $X$ be an integral noetherian scheme and $Z\subset X$ be a non-empty, closed subscheme. We say that $Z$ is \emph{G3} in $X$ if the natural embedding $$K(X)\hookrightarrow K(X_{/Z})$$
from the ring of rational functions to the ring of formal rational functions is an isomorphism.
\end{defn}
The following theorem is an easy consequence of \cite[Ch V]{hartshorneample}.
\begin{thm}\label{Theorem:algebraize_formal_map}Let $X/k$ be a smooth projective variety. Let $C\subset X$ be a smooth curve which is the complete intersection of smooth, ample divisors. Let $\hat{f}\colon X_{/C}\rightarrow M$ be a morphism to a quasi-projective variety. Then there exists an open set $W\supset C$ together with a map $f\colon W\rightarrow M$ such that $\hat{f}$ is the formal completion of $f$.
\end{thm}
\begin{proof}
It follows from \cite[Ch V, Corollary 2.3, p. 202]{hartshorneample} that $C$ is G3 in $X$. Recall that topologically, $X_{/C}$ is homeomorphic to $C$; when we make set-theoretic arguments, we pass freely between $X_{/C}$ and $C$. Embed $M$ in $\PP^m_k$ and $\hat{L}:=\hat{f}^*\mathcal{O}_M(1)$ be the pullback of the ample class to $X_{/C}$. Then $\hat{L}$ is globally generated. Pick regular sections $(\hat{s}_i)^m_{i=1}$ of $H^0(X_{/C},\hat{L})$ that define the map to $\PP^m$; in particular, the $\hat{s}_i$ globally generate. Let $V_i\subset X_{/C}$ be the locus on which $\hat{s}_i\neq 0$. Then $\tilde{s}_i$ yields a trivialization over $V_i$:
$$\hat{s}_i\colon \calO_{X_{/C}}|_{V_i}\rightarrow \hat{L}_{V_i}.$$
As the $s_i$ globally generate $\hat{L}$, it follows that $\displaystyle \cup_i V_i=X_{/C}$.

With respect to the trivializations given by the $\hat{s}_i$, the line bundle $\hat{L}$ is defined by the transition functions $\hat{\phi}_{ij}:=\frac{s_i}{s_j}$ on $V_{ij}:=V_i \cap V_j$. By definition, $\hat{\phi}_{ij}\in K(X_{/C})$. By the G3 property, there is a \emph{unique} meromorphic function $\phi_{ij}\in K(X)$ that restricts to $\hat{\phi}_{ij}$ on $X_{/C}$. Call $D_{ij}$ the divisor of zeroes of $\phi_{ij}$.

For each $i\in \{1,2,\ldots,m\}$, there exists an open set $U_i\subset X$ such that
\begin{itemize}
\item $U_i\cap X_{/C}=V_i$ and
\item $U_i \cap D_{ij} = \emptyset$ for $j\neq i$.
\end{itemize}
Indeed, we can of course find an open set $U_i\subset X$ with $U_i\cap X_{/C}=V_i$. Replace $U_i$ by the complement of $\displaystyle\cup_{j\neq i}D_{ij}\cap U_i$ in $U_i$; in other words, remove all of the zeroes of the $\phi_{ij}$ from $U_i$ for $j\neq i$. The $\hat{s}_i$ are regular (and not merely meromorphic) sections; hence $D_{ij}\cap X_{/C}$ is a subset of the finite set of zeros of $\hat{s_i}$. Therefore $U_i\cap X_{/C}=V_i$. Set
$$\displaystyle W:=\cup_i U_i \supset X_{/C}$$

and $U_{ij}:=U_i\cap U_j$. Then on $U_{ij}$, the functions $\phi_{ij}$ are invertible. Moreover, the fact that $\hat{\phi}_{ij}\hat{\phi}_{jk}=\hat{\phi}_{ik}$ implies that $\phi_{ij}\phi_{jk}=\phi_{ik}$ by the injectivity of the map $K(X)\hookrightarrow K(X_{/C})$. Therefore we obtain a line bundle $L_W$ on $W$ algebraizing $\hat{L}$. 

By $Lef(X,C)$ \cite[Ch. V, Prop 2.1, p. 200]{hartshorneample}, we deduce that after shrinking $W$ (while still containing $C$), we may ensure that $H^0(W,L_W)=H^0(X_{/C},\hat{L})$. This implies that the $\hat{s}_i$ algebraize to sections $s_i$ of $H^0(W,L_W)$. Further replace $W$ by the complement of the common locus of zeroes of the $s_i$.

Finally, we obtain a morphism $f\colon W\rightarrow \PP^m$. As $f|_{X_{/C}}$ lands in $M$ and $X$ is integral, this implies that the image of $g$ entirely lands in the Zariski closure $\bar{M}$ of $M\subset \PP^m$ by schematic density of $X_{/C}$ inside of $X$. By further shrinking $W$, we obtain that the map $\hat{f}\colon X_{/C}\rightarrow M$ uniquely algebraizes to a map $f\colon W\rightarrow M$, as desired. 
\end{proof}
\begin{cor}\label{Corollary:algebraize_formal_abelian}Let $X/k$ be a smooth projective variety and let $C\subset X$ be a smooth curve that is the complete intersection of smooth ample divisors. Let $\hat{A}\rightarrow X_{/C}$ be an polarizable abelian scheme over the formal completion of $X$ along $C$. If there exists $l>2$, prime to $\text{char}(k)$, such that the $l$-torsion of $\hat{A}|_C\rightarrow C$ is trivial, then there exists a Zariski open subset of $U\subset X$ containing $C$, and an abelian scheme $A_U\rightarrow U$ that extends $\hat{A}\rightarrow X_{/C}$.
\end{cor}
\begin{proof}
Note that by topological invariance of the \'etale site, if the $l$-torsion of $\hat{A}|_C$ is trivial, then so is the $l$-torsion of $\hat{A}$. If $l>2$, then the moduli space $\mathscr{A}_{g,d,l}$ exists as a quasi-projective scheme over $\Spec(\ZZ[1/l])$. Apply Theorem \ref{Theorem:algebraize_formal_map} to the induced map $X_{/C}\rightarrow \mathscr{A}_{g,d,l}$. 
\end{proof}
Corollary \ref{Corollary:algebraize_formal_abelian} is not true if one does not assume that the $l$-torsion is trivial for the same reason that Theorem \ref{Theorem:algebraize_formal_map} is not true with $M$ a Deligne-Mumford stack. For example, if $C\subset \PP^2$ is a smooth plane curve of positive genus, it is easy to see that there is a non-trivial map $\PP^2_{/C}\rightarrow B\ZZ/2\ZZ$ that does not algebraize.

D. Litt explained the following lemma to the first author. 
\begin{lem}
\label{Lemma_BT_extending_iso}Let $X/k$ be a smooth projective variety with $\dim(X)\geq 2$ and let $U\subset X$ be an open subset. Let $C\subset U$ be a smooth curve that is the complete complete intersection of smooth ample divisors of $X$. Let $G_{U}$ and $H_{U}$ be finite flat group schemes or Barsotti-Tate groups on $U$. Then the natural restriction map:
\begin{equation}\label{equation:restrict_maps_BT}\displaystyle
\text{Isom}_{U}(G_U,H_U)\rightarrow \text{Isom}_{X_{/C}}(G_U|_{X_{/C}},H_U|_{X_{/C}})
\end{equation}
is an isomorphism of sets. In particular, if $G_U|_{X_{/C}}$ and $H_U|_{X_{/C}}$ are isomorphic, then $G_U$ and $H_U$ are isomorphic.
\end{lem}

\begin{proof}
As usual, $U$ has complementary codimension at least 2. If we prove the lemma for finite flat group schemes, then the result for BT groups immediately follows: a homomorphism of BT groups is simply a compatible sequence of homomorphisms of $n$-trucated BT groups. Therefore we assume $G_{U}$ and $H_{U}$ are finite flat group schemes, in which case Equation \ref{equation:hom_BT_eta} is clearly injective by the schematic density of $X_{/C}$ in $X$. Our task is therefore to prove that Equation \ref{equation:restrict_maps_BT} is surjective.

 We remark that the $Lef(X,C)$ is satisfied by \cite[Ch. V, Proposition 2.1]{hartshorneample}. By definition, this means for any vector bundle $E$ on a Zariski neighborhood
$V\supset D$, the natural restriction map on sections
\[
H^{0}(V,E)\cong H^{0}(X_{/C},E|_{X_{/C}})
\]
is an isomorphism. Now, finite flat group schemes are given by their associated (coherent) sheaves of Hopf algebras. Write $\mathcal{O}_{G_U}$ and $\mathcal{O}_{H_U}$ for the Hopf algebras of $G_U$ and $H_U$ respectively. Then $Lef(X,C)$ implies that any isomorphism of formal vector bundles $$\hat{\varphi}\colon \mathcal{O}_{G_{U}}|_{X_{/C}}\rightarrow \mathcal{O}_{H_{U}}|_{X_{/C}}$$ canonically lifts to an isomorphism $\varphi\colon \mathcal{O}_{G_U}\rightarrow \mathcal{O}_{H_U}$ of vector bundles on $U$. We must show if $\hat{\varphi}$ preserves the Hopf-algebra structure, then $\varphi$ also preserves the Hopf-algebra structure. As the Hopf-algebra structures on are given as morphisms of vector bundles, the fact that $\hat{\varphi}$ is an isomorphism of sheaves of Hopf algebras together with $Lef(X,C)$ implies that $\varphi$ is also an isomorphism of sheaves of Hopf algebras on $U$, as desired.

\end{proof}
As a byproduct of the algebraization machinery, we have the following, which is a cousin of Simpson's Theorem \ref{Theorem:simpson_lefschetz}.
\begin{cor} \label{Corollary:AV_extends_BT_extends}Let $X/\Fq$ be a smooth projective variety with $\dim X\geq2$ and let $U\subset X$ be an open subset whose complement has codimension at least 2. Let $C\subset U$ be a smooth curve that is the complete intersection of smooth, ample divisors of $X$. Let $A_{C}\rightarrow C$
be an abelian scheme of dimension $g$. Then the following are equivalent.
\begin{itemize}
\item There exists a Barsotti-Tate group $\calG_U$ on $U$ such that $\calG_{C}\cong A_C[p^{\infty}]$.
\item There exists an abelian scheme $A_U\rightarrow U$ extending $A_C$.
\end{itemize}
\end{cor}

\begin{proof}
It suffices to prove the first statement implies the second. By Corollary \ref{Corollary:formal_abelian_scheme}, there exists a (polarizable) formal abelian scheme $\hat{A}$ over $X_{/C}$ extending $A_{C}.$ Picking a polarization, this yields a map $X_{/D}\rightarrow\mathscr{A}_{g,d}$. 

By \cite[Theorem 5.1]{kedlaya2016notes}, the restriction functor $\fisoc{X}\rightarrow \fisoc{U}$ is an equivalence of categories. Denote the canonical extension of $\mathbb{D}(\calG_U)\otimes \QQ\in\fisoc{U}$ by $\calE\in\fisoc{X}$. Our assumption implies that $\calE_C\cong \mathbb{D}(A_C[p^{\infty}])\otimes \QQ.$
Using the equivalence of categories and \cite[Theorem 1.2 and Proposition 4.9]{pal2015monodromy}, it follows that $\calE$ is semi-simple.

As $\det(\calE|_C)\cong \Qp(-g)$ it follows that $\det(\calE)\cong \Qp(-g)$. Every irreducible summand $\calE_i$ of $\calE$ is algebraic because it is on $C$. We note that for $c$ a closed point of $C$, we have $P_c(\calE,t)\in \QQ[t]$. 

Pick an isomorphism $\iota\colon \Qpbar\rightarrow \Qlbar$. By \cite[Theorem 0.3]{abe2016lefschetz} or \cite[Theorem 0.4.1]{kedlayacompanions}, there is a unique semi-simple $\iota$-companion to $\calE$, which is a lisse $\Qlbar$ sheaf $L$.
The $l$-adic sheaf $L$ may be considered as a representation $W(X,\bar{c})\rightarrow \text{GL}_{2g}(\Qlbar)$. The map $W(C,\bar{c})\rightarrow W(X,\bar{c})$ is surjective by the usual Lefschetz theorems and the representation $W(C,\bar{c})\rightarrow \text{GL}_{2g}(\Qlbar)$ conjugates into a representation $W(C,\bar{c})\rightarrow \text{GL}_{2g}(\ZZ_l)$ because $L_C$ is compatible with $A_C\rightarrow C$. Hence the representation $L$ conjugates into a $\ZZ_l$ local system on $X$. Take a connected finite \'etale cover of $X'$ of $X$ that trivializes the mod-$l$ representation. Let $C'$ and $U'$ denote the inverse image to $X'$ of $C$ and $U$ respectively, and note that $C'$ is still the smooth complete intersection of smooth ample divisors inside of $X'$. Then we obtain a map $X'_{/C'}\rightarrow \mathscr{A}_{g,d,l}$, which globalizes to an open neighborhood $V'\subset X'$ of $C'$ by Corollary \ref{Corollary:algebraize_formal_abelian}. 
Lemma \ref{Lemma_BT_extending_iso} implies that $A_{V'}[p^{\infty}]\cong\mathcal{G}_{V'}$
because they agree on $X'_{/C'}$. By Corollary \ref{Corollary:grothendieck},
$A_{V'}$ extends to an abelian scheme $A_{U'}\rightarrow U'$ with $A_{U'}[p^{\infty}]\cong\mathcal{G}_{U'}$.

Finally, note that the abelian scheme $A_{U'}\rightarrow U'$, equipped with isomorphisms $A_{U'}|_{C'}\cong A_{C'}$ and $A_{U'}[p^{\infty}]\cong \calG_{U'}$ is \emph{unique up to unique isomorphism}. Indeed, use \cite[Proposition 1.2 on p. 64]{grothendieck1966theoreme}. Set $S=X$, let $s$ be the generic point of $X$ (which is the unique associated point of $S$), let $s'$ be any point of $C$, and let $S'$ be $\Spec(\hat{\calO}_{S,s'})$ be the completion of the local ring at $s'$. Then $S'\rightarrow \Spec(\calO_{S,s'})$ is faithfully flat, and the Serre-Tate theorem will imply the desired uniqueness.

Moreover, $A_{U'}\rightarrow U'$ is projective. Therefore we may apply faithfully flat descent to canonically descend to an abelian scheme $A_U\rightarrow U$, as desired. (This argument appears in \cite[p. 73, lines 7-12]{grothendieck1966theoreme}.)
\end{proof}
\begin{rem}We note that a slightly stronger variant of Corollary \ref{Corollary:AV_extends_BT_extends} is true when $\Fq$ is replaced by a field $k$ of characteristic 0. A Barsotti-Tate group in characteristic 0 is \'etale and hence satisfies Zariski-Nagata purity. One uses Simpson's Theorem \ref{Theorem:simpson_lefschetz} to prove that there is a family $A_{X,\CC}\rightarrow X_{\CC}$ extending $A_{C,\CC}\rightarrow C_{\CC}$. Then one transports the descent datum of $A_{C,\CC}$ to $A_{X,\CC}$ using Theorem \ref{Theorem:Lefschetz_AV}. Finally, applying Weil descent \cite[Appendix A]{varshavskyshimura} shows that we may find an abelian scheme $A_X\rightarrow X$ extending $A_C\rightarrow C$, as desired.


\end{rem}

\begin{rem}\label{Remark:extend_AV_general_field}
We note that the key point where we use that the base field is $\Fq$ in Corollary \ref{Corollary:AV_extends_BT_extends} is to prove that a specific connected finite \'etale cover of $C$ extends to a connected finite \'etale cover of $X$. We do this via $l$-adic companions.
\end{rem}
The following is a version of Corollary \ref{Corollary:AV_extends_BT_extends} over general fields of characteristic $p$. Note that, in contrast to Corollary \ref{Corollary:AV_extends_BT_extends}, we assume that the subvariety has dimension at least 2.
\begin{cor}\label{Corollary:extend_AV_extend_BT_ii}
Let $X/k$ be a smooth projective variety over a field of characteristic $p$ with $\dim X\geq 3$ and let $U\subset X$ be an open subset whose complement has codimension at least 3. Let $Z\subset U$ be a smooth complete subvariety of dimension at least 2 that is the complete intersection of smooth, ample divisors of $X$. Let $A_{Z}\rightarrow Z$
be an abelian scheme of dimension $g$. Then the following are equivalent.
\begin{itemize}
\item There exists a Barsotti-Tate group $\calG_U$ on $U$ such that $\calG_{Z}\cong A_Z[p^{\infty}]$.
\item There exists an abelian scheme $A_U\rightarrow U$ extending $A_Z$.
\end{itemize}
\end{cor}
\begin{proof}First of all, we explain the dimension constraints. Let $2\leq d=\dim(Z)$. Then the existence of an open set $U$ such that $Z\subset U$ implies that the codimension of the complement of $U$ in $X$ is at least $d+1\geq 3$ because $Z$ is proper and the complete intersection of ample divisors. Moreover, we claim that the maps $\pi_1(Z)\rightarrow  \pi_1(U)\rightarrow \pi_1(X)$ are all isomorphisms. Indeed, by the classical Lefschetz theorems, $\pi_1(Z)\rightarrow \pi_1(X)$ is an isomorphism because $\dim(Z)\geq 2$ \cite[Ch. IV, Theorem 1.5, Theorem 2.1]{hartshorneample}. Moreover, the map $\pi_1(U)\rightarrow \pi_1(X)$ is an isomorphism by Zariski-Nagata purity. Therefore, a connected finite \'etale cover $Z'\rightarrow Z$ trivializing the $l$-torsion of $A_Z\rightarrow Z$ extends canonically to a connected finite \'etale cover $X'\rightarrow X$. Then the exact proof of Corollary \ref{Corollary:AV_extends_BT_extends} applies. 
\end{proof}

\begin{rem}
\label{Remark:extension_dimension_bounds}Corollary \ref{Corollary:AV_extends_BT_extends} is not true if one assumes that $\dim(X)=1$ and $C$ is a point. For example, let $\mathcal{X}$ be the good reduction modulo $p$ of a moduli space of fake elliptic curves, together with a universal family $\mathcal{A}\rightarrow\mathcal{X}$ of abelian surfaces with quaternionic multiplication. There is a decomposition $\mathcal{A}[p^{\infty}]\cong\mathcal{G}\oplus\mathcal{G}$, where $\mathcal{G}$ is a height 2, dimension 1, everywhere versally deformed BT group $\mathcal{G}$. Over any point $x\in\mathcal{X}(\mathbb{F})$,
there exists an elliptic curve $E/\mathbb{F}$ such that $E[p^{\infty}]\cong\mathcal{G}_{x}$. On the other hand, the induced formal deformation of the elliptic curve $E\rightarrow x$ certainly does not extend to a non-isotrivial elliptic curve over $X$.
\end{rem}

We now have all of the ingredients to prove our main Lefschetz theorem.
\begin{thm}\label{Theorem:main_lefschetz}Let $X/\Fq$ be a smooth projective variety. Then there exists an open subset $U\subset X$, whose complement has codimension at least 2, such that the following holds.

Let $C\subset U$ be a smooth projective curve that is the complete intersection of smooth ample divisors of $X$. Let $\pi_C\colon A_C\rightarrow C$ be an abelian scheme of $\text{GL}_2$-type: for a prime $l\neq p$, $R^1(\pi_C)_*\Qlbar$ has irreducible summands that have rank 2 and determinant $\Qlbar(-1)$. Then the following are equivalent.

\begin{itemize}
\item There exists an abelian scheme of $\text{GL}_2$-type $B_U\rightarrow U$ with $B_C\rightarrow C$ isogenous to $A_C\rightarrow C$.
\item The $F$-isocrystal $\mathbb{D}(A_C[p^{\infty}])\otimes \Qpbar\in \fisoc{C}_{\Qpbar}$ extends to an $F$-isocrystal $\calE\in \fisoc{X}_{\Qpbar}$.
\end{itemize}
\end{thm}
\begin{proof}

We first construct $U$. It follows from \cite[Corollary 4.3]{abe2016lefschetz} that there are only finitely many isomorphism classes of irreducible overconvergent $F$-isocrystals $\calF_i$ on $X$ of rank 2 and with determinant $\Qpbar(-1)$. We may pick a $p$-adic local field $K$ over which each such $\calF_i$ is defined. By the slope bounds of Theorem \ref{Theorem:rank_2_slope_bounds}, the two slopes of each such $\calF_i$ differ by at most one. As the two slopes add to one, this implies that the two slopes are in the interval $[0,1]$. (This is the key place where we use the assumption that the $\calF_i$ have rank 2.) Then Lemma \ref{Lemma:Katz} implies that there is an open subset $U_i\subset X$, whose complement has codimension at least 2 on which $\calF_i$ underlies a Dieudonn\'e $\mathscr{M}_i$ in finite, locally free modules with $\mathscr{M}_i\otimes \Qp \cong \calF_i|_{U_i}$ (where on the left hand side we forget $V$ and on the right hand side we forget the $K$-structure). Let $U:=\cap U_i$. In particular \emph{every} overconvergent $F$-isocrystal on $X$ that is irreducible of rank 2 and has determinant $\Qpbar(-1)$ underlies a Dieudonn\'e crystal on $U$. This is $U$ in the first part of the theorem.

The first condition implies the second: $\mathbb{D}(B_U[p^{\infty}])\otimes \Qp\in \fisocd{U}$, and there is an equivalence of categories $\fisoc{X}\rightarrow \fisoc{U}$ because $U$ has complementary codimension at least 2 \cite[Theorem 5.1]{kedlaya2016notes}. Set $\calE$ to be the extension of $\mathbb{D}(B_U[p^{\infty}])\otimes \Qpbar$ to $\fisoc{X}_{\Qpbar}$.

Now, let $\calE$ be as in the theorem. We make three claims about $\calE$.
\begin{enumerate}

\item For every closed point $x$ of $X$, the polynomial $P_x(\calE,t)$ has coefficients in $\QQ$.
\item The object $\calE$ is semisimple and the irreducible summands of $\calE$ have rank 2 and determinant $\Qpbar(-1)$.
\item The object $\calE$ may be descended to $\fisoc{X}$; that is, it has coefficients in $\Qp$.

\end{enumerate}

To prove the first part, pick an isomorphism $\iota\colon \Qpbar\rightarrow \Qlbar$. By \cite[Theorem 0.3]{abe2016lefschetz} or \cite[Theorem 0.4.1]{kedlayacompanions}, there is a unique semi-simple $\iota$-companion to $\calE$, which is a lisse $\Qlbar$ sheaf $L$. For $c$ a closed point of $C$, the polynomial $P_c(L,t)\in \QQ[t]$ by the companion relation. Applying \cite[Proposition 7.1]{esnault2011notes} and induction, we deduce that for \emph{all} closed points $x$ of $X$,
$$P_x(L,t)\in \QQ[t]$$
and hence $P_x(\calE,t)\in \QQ[t]$.

To prove the second part, first note that $\calE_C\in \fisoc{X}_{\Qpbar}$ has irreducible summands of rank 2 and determinant $\Qpbar(-1)$. This follows because $\mathbb{D}(A_C[p^{\infty}])\otimes \Qpbar$ is isomorphic to $\Qpbar$-rational crystalline cohomology, which is a companion to  $R^1(\pi_C)_*\Qlbar$ by \cite{katzmessing}. Moreover, $\mathbb{D}(A_C[p^{\infty}])\otimes \Qpbar$ is semisimple by \cite{pal2015monodromy}. Then the fact that $R^1(\pi_C)_*\Qlbar$ has irreducible summands of rank 2 with cyclotomic determinant implies the same result for $\calE_C$. Finally, as $C$ is the smooth complete intersection of smooth ample divisors, Theorem \ref{Theorem:AE_Lefschetz} implies that $\calE$ is semisimple and the irreducible summands of $\calE$ have rank 2 and cyclotomic determinant.

To prove the third part, for every $\sigma\in \text{Aut}_{\Qp}(\Qpbar)$, as $P_x(\calE,t)\in\QQ[t]$ for all closed points $x$ of $X$, Theorem \ref{Theorem:Brauer_Nesbitt_coefficients} implies that $^{\sigma}(\calE)\cong \calE$. To prove that $\calE$ descends to $\fisoc{X}$, we must construct isomorphisms $c_{\sigma}\colon ^{\sigma}(\calE)\cong \calE$ that satisfy the cocycle condition by \cite[Lemma 4]{sosna2014scalar}. Now, the natural map
$$\Hom(^{\sigma}\calE,\calE)\rightarrow \Hom(^{\sigma}\calE_C,\calE_C)$$
is an isomorphism by Theorem \ref{Theorem:AE_Lefschetz}. Note that $\calE_C$ descends to $\fisoc{C}$: indeed, $\mathbb{D}(A_C[p^{\infty}])\otimes \QQ\in \fisoc{C}$. Therefore, the descent data transports to the left and hence $\calE\in \fisoc{X}$.

By construction of $U$, the $F$-isocrystal $\calE$ underlies a Dieudonn\'e crystal $(M,F,V)$ on $U$ whose associated $F$-isocrystal is isomorphic to $\calE_U$. By Theorem \ref{Theorem:de Jong}, there is an associated $p$-divisible group $\calG_U$, with $\mathbb{D}(\calG_U)\cong (M,F,V)$. It then follows that $A_C[p^{\infty}]$ is isogenous to $\calG_C$: their Dieudonn\'e crystals are isomorphic when viewed as $F$-isocrystals. Let $K\subset A_C[p^{\infty}]$ be the kernel of one such isogeny. Then $B_C:=A_C/K$ is an abelian scheme with $B_C[p^{\infty}]\cong \calG_C$. Applying Corollary \ref{Corollary:AV_extends_BT_extends}, we see there exists an abelian scheme $B_U\rightarrow U$ (of dimension $g$) as desired. As $B_U[p^\infty]$ and $\calG_U$ are BT groups on $U$ that are isomorphic on $X_{/C}$, there are isomorphic by Lemma \ref{Lemma_BT_extending_iso}. Therefore $B_U$ is compatible with $\calE_U$, as desired.

\end{proof}
In the proof of Theorem \ref{Theorem:main_lefschetz}, the key step where we use the fact that the irreducible summands have rank 2 is to deduce that the Newton slopes of $\calE$ are in the range $[0,1]$ everywhere on $X$.

It is natural to wonder whether or not one can remove both the open subset and the isogeny in Theorem \ref{Theorem:main_lefschetz}, in further analogy with Theorem \ref{Theorem:simpson_lefschetz}. Johan de Jong suggested the following example, which shows that the answer is ``no'' for general abelian varieties.
\begin{example}Let $E/\Fq$ be a supersingular elliptic curve, so $\alpha_p\times\alpha_p\subset E\times E$. Over $\AA^2\backslash (0,0)$, there is an injective homomorphism
$$\iota\colon (\alpha_p)_{\AA^2\backslash (0,0)}\hookrightarrow (\alpha_p\times\alpha_p)_{\AA^2\backslash (0,0)}\subset(E\times E)_{\AA^2\backslash (0,0)}$$
given as follows: let $(s,t)$ be linear coordinates on $\alpha_p\times \alpha_p\subset \AA^2$ and let $(x,y)$ be the coordinates on $\AA^2\backslash (0,0)$. Then the above sub-group scheme is defined by the equation $$[x:y]=[s:t].$$
Let $A$ be the following (principally polarizable) supersingular abelian surface over $\AA^2\backslash (0,0)$:
$$A:=(E\times E)_{\AA^2\backslash (0,0)}/\iota(\alpha_p)$$
\cite{moret1981familles}. (The principal polarization might only exist after a finite field extension.) Let $L\subset \AA^2$ be a line not passing through the origin. Consider the compactification $\AA^2\backslash (0,0)\subset \PP^2$ and let $\bar L$ be the closure of $L$ in $\PP^2$. Then $A_L$ extends to a non-constant supersingular abelian scheme $A_{\bar L}\rightarrow \bar L$ and $\bar L\subset \PP^2$ is a smooth, ample divisor. However, it is easy to see that $A_{\bar L}$ cannot extend to $\PP^2$.  If it did, the family would have to be supersingular as the geometric monodromy is trivial and there is at least one supersingular fiber. Let $\mathscr{A}_{2,1}\otimes\FF$ be the moduli of principally polarized abelian surfaces over $\FF$. Then the supersingular locus of is a union of complete rational curves \cite[p. 177]{oort1974} and there are no morphisms $\PP^2\rightarrow \PP^1$.
\end{example}

\begin{cor}\label{Corollary:rank_2_reduces_to_curves}
Let $X/\Fq$ be a smooth projective variety with $\dim(X)\geq 2$ and let $L$ be a rank 2 lisse $\Qlbar$ sheaf with cyclotomic determinant and infinite geometric monodromy. Then there exists an open subset $U\subset X$, whose complement has codimension at least 2, such that
\begin{itemize}
\item if $C\subset U$ is a smooth proper curve that is the complete intersection of smooth ample divisors;
\item if $L_C$ comes from an abelian scheme on $A_C\rightarrow C$, in the sense of Conjecture \ref{Conjecture:R2}; and
\item if all $p$-adic companions to $L$ exist,
\end{itemize} then $L_U$ comes from an abelian scheme $B_U\rightarrow U$, i.e., Conjecture \ref{Conjecture:R2} is true for $(X,L)$
\end{cor}
\begin{proof}
By the usual Lefschetz theorems, $W(C)\rightarrow W(X)$ is surjective. Therefore if two lisse $\Qlbar$ sheaves on $X$ are isomorphic on $C$, then they are isomorphic. The $U$ in our theorem is the $U$ guaranteed to exist by Theorem \ref{Theorem:main_lefschetz}. 

Every irreducible summand of (the semi-simple lisse sheaf) $R^1\pi_{C*}\Qlbar$ is a companion of $L_C$ by assumption, and hence has determinant $\Qlbar(-1)$. Therefore each of the $\calE_i$ has determinant $\Qpbar(-1)$.  Consider $ \mathbb{D}(A_C[p^{\infty}])\otimes \Qpbar\in \fisoc{X}_{\Qpbar}$. This is semisimple by \cite{pal2015monodromy}. There is a decomposition:
$$\displaystyle \mathbb{D}(A_C[p^{\infty}])\otimes \Qpbar \cong \bigoplus_i(\calE_i|_C)^{m_i}$$
for some integers $m_i\geq 1$ because the LHS is isomorphic to relative crystalline cohomology \cite{katzmessing}. Set $\calE:=\bigoplus_i (\calE_i)^{m_i}$ and apply Theorem \ref{Theorem:main_lefschetz}.
\end{proof}

\bibliographystyle{alphaurl}
\bibliography{extending_abelian}

\newcommand{\etalchar}[1]{$^{#1}$}
\begin{thebibliography}{FWG{\etalchar{+}}92}

\bibitem[Abe18a]{abe2013langlands}
Tomoyuki Abe.
\newblock Langlands correspondence for isocrystals and the existence of
  crystalline companions for curves.
\newblock {\em J. Amer. Math. Soc.}, 31(4):921--1057, 2018.
\newblock \href {http://dx.doi.org/10.1090/jams/898}
  {\path{doi:10.1090/jams/898}}.

\bibitem[Abe18b]{abe2011langlands}
Tomoyuki Abe.
\newblock Langlands program for {$p$}-adic coefficients and the petits
  camarades conjecture.
\newblock {\em J. Reine Angew. Math.}, 734:59--69, 2018.
\newblock \href {http://dx.doi.org/10.1515/crelle-2015-0045}
  {\path{doi:10.1515/crelle-2015-0045}}.

\bibitem[AC13]{abe2013theory}
Tomoyuki Abe and Daniel Caro.
\newblock Theory of weights in $p$-adic cohomology.
\newblock {\em arXiv preprint arXiv:1303.0662}, 2013.

\bibitem[AE19]{abe2016lefschetz}
Tomoyuki Abe and H{\'{e}}l{\`{e}}ne Esnault.
\newblock A {L}efschetz theorem for overconvergent isocrystals with frobenius
  structure.
\newblock {\em Ann. Scient. {\'{E}}co. Norm. Sup.}, 52(5):1243--1264, 2019.
\newblock \href {http://dx.doi.org/10.24033/asens.2408}
  {\path{doi:10.24033/asens.2408}}.

\bibitem[BBM82]{berthelot1982theorie}
Pierre Berthelot, Lawrence Breen, and William Messing.
\newblock {\em Th{\'e}orie de Dieudonn{\'e} cristalline II}, volume 930 of {\em
  Lecture Notes in Mathematics}.
\newblock Springer-Verlag, Berlin, 1982.
\newblock \href {http://dx.doi.org/10.1007/BFb0093025}
  {\path{doi:10.1007/BFb0093025}}.

\bibitem[CCO14]{chai2014complex}
Ching-Li Chai, Brian Conrad, and Frans Oort.
\newblock {\em Complex multiplication and lifting problems}, volume 195 of {\em
  Mathematical Surveys and Monographs}.
\newblock American Mathematical Society, Providence, RI, 2014.

\bibitem[Chi04]{chin2004independence}
CheeWhye Chin.
\newblock Independence of {$l$} of monodromy groups.
\newblock {\em J. Amer. Math. Soc.}, 17(3):723--747, 2004.
\newblock \href {http://dx.doi.org/10.1090/S0894-0347-04-00456-4}
  {\path{doi:10.1090/S0894-0347-04-00456-4}}.

\bibitem[CO09]{chai2009moduli}
Ching-Li Chai and Frans Oort.
\newblock Moduli of abelian varieties and {$p$}-divisible groups.
\newblock 8:441--536, 2009.

\bibitem[Con06]{conrad2006chow}
Brian Conrad.
\newblock {Chow's $K/k$-image and $K/k$-trace, and the Lang-N\'eron theorem}.
\newblock {\em Enseign. Math. (2)}, 52(1-2):37--108, 2006.

\bibitem[Cre92]{crew1992f}
Richard Crew.
\newblock {$F$}-isocrystals and their monodromy groups.
\newblock {\em Ann. Sci. \'Ecole Norm. Sup. (4)}, 25(4):429--464, 1992.
\newblock URL: \url{http://www.numdam.org/item?id=ASENS_1992_4_25_4_429_0}.

\bibitem[CS08]{corlette2008classification}
Kevin Corlette and Carlos Simpson.
\newblock On the classification of rank-two representations of quasiprojective
  fundamental groups.
\newblock {\em Compos. Math.}, 144(5):1271--1331, 2008.
\newblock \href {http://dx.doi.org/10.1112/S0010437X08003618}
  {\path{doi:10.1112/S0010437X08003618}}.

\bibitem[Del80]{deligne1980conjecture}
Pierre Deligne.
\newblock La conjecture de {W}eil. {II}.
\newblock {\em Inst. Hautes \'Etudes Sci. Publ. Math.}, (52):137--252, 1980.
\newblock URL: \url{http://www.numdam.org/item?id=PMIHES_1980__52__137_0}.

\bibitem[Del12]{deligne2012finitude}
Pierre Deligne.
\newblock Finitude de l'extension de {$\Bbb Q$} engendr\'ee par des traces de
  {F}robenius, en caract\'eristique finie.
\newblock {\em Mosc. Math. J.}, 12(3):497--514, 668, 2012.

\bibitem[Del14]{deligne2014semi}
Pierre Deligne.
\newblock Semi-simplicit{\'e} de produits tensoriels en caract{\'e}ristique p.
\newblock {\em Inventiones mathematicae}, 197(3):587--611, 2014.

\bibitem[dJ95]{de1995crystalline}
A.~J. de~Jong.
\newblock Crystalline {D}ieudonn\'e module theory via formal and rigid
  geometry.
\newblock {\em Inst. Hautes \'Etudes Sci. Publ. Math.}, (82):5--96 (1996),
  1995.
\newblock URL: \url{http://www.numdam.org/item?id=PMIHES_1995__82__5_0}.

\bibitem[dJ98]{de1998homomorphisms}
A.~J. de~Jong.
\newblock Homomorphisms of {B}arsotti-{T}ate groups and crystals in positive
  characteristic.
\newblock {\em Invent. Math.}, 134(2):301--333, 1998.
\newblock \href {http://dx.doi.org/10.1007/s002220050266}
  {\path{doi:10.1007/s002220050266}}.

\bibitem[dJM99]{de1999crystalline}
A.~J. de~Jong and W.~Messing.
\newblock Crystalline {D}ieudonn\'e theory over excellent schemes.
\newblock {\em Bull. Soc. Math. France}, 127(2):333--348, 1999.
\newblock URL: \url{http://www.numdam.org/item?id=BSMF_1999__127_2_333_0}.

\bibitem[DK17]{drinfeld2016slopes}
Vladimir Drinfeld and Kiran~S. Kedlaya.
\newblock Slopes of indecomposable {$F$}-isocrystals.
\newblock {\em Pure and Applied Mathematics Quarterly}, 13(1):131--192, 2017.
\newblock \href {http://dx.doi.org/10.4310/pamq.2017.v13.n1.a5}
  {\path{doi:10.4310/pamq.2017.v13.n1.a5}}.

\bibitem[{Dri}77]{drinfeld1977}
V.~G. {Drinfel'd}.
\newblock {Elliptic modules. II.}
\newblock {\em {Math. USSR, Sb.}}, 31:159--170, 1977.

\bibitem[{Dri}83]{drinfeld1983}
V.~G. {Drinfeld}.
\newblock {Two-dimensional \(\ell\)-adic representations of the fundamental
  group of a curve over a finite field and automorphic forms on
  \(\mathrm{GL}(2)\).}
\newblock {\em {Am. J. Math.}}, 105:85--114, 1983.

\bibitem[Dri12]{drinfeld2012conjecture}
Vladimir Drinfeld.
\newblock {On a conjecture of Deligne}.
\newblock {\em Moscow Mathematical Journal}, 12(3):515--542, 2012.

\bibitem[EG17]{esnault2017cohomologically}
H{\'e}l{\`e}ne Esnault and Michael Groechenig.
\newblock Cohomologically rigid local systems and integrality.
\newblock {\em arXiv preprint arXiv:1711.06436}, 2017.

\bibitem[Eis95]{eisenbud2013commutative}
David Eisenbud.
\newblock {\em Commutative algebra}, volume 150 of {\em Graduate Texts in
  Mathematics}.
\newblock Springer-Verlag, New York, 1995.
\newblock With a view toward algebraic geometry.
\newblock \href {http://dx.doi.org/10.1007/978-1-4612-5350-1}
  {\path{doi:10.1007/978-1-4612-5350-1}}.

\bibitem[EK11]{esnault2011notes}
H{\'e}l{\`e}ne Esnault and Moritz Kerz.
\newblock {Notes on Deligne's letter to Drinfeld dated March 5, 2007}.
\newblock {\em Notes for the Forschungsseminar in Essen, summer}, 2011.
\newblock Available at
  \url{http://userpage.fu-berlin.de/esnault/preprints/helene/103-110617.pdf}.

\bibitem[EK12]{esnault2012finiteness}
H{\'e}l{\`e}ne Esnault and Moritz Kerz.
\newblock {A finiteness theorem for Galois representations of function fields
  over finite fields (after Deligne)}.
\newblock {\em arXiv preprint arXiv:1208.0128}, 2012.

\bibitem[Esn17]{esnault2017survey}
H\'el\`ene Esnault.
\newblock Survey on some aspects of {L}efschetz theorems in algebraic geometry.
\newblock {\em Rev. Mat. Complut.}, 30(2):217--232, 2017.
\newblock \href {http://dx.doi.org/10.1007/s13163-017-0223-8}
  {\path{doi:10.1007/s13163-017-0223-8}}.

\bibitem[FC90]{faltings2013degeneration}
Gerd Faltings and Ching-Li Chai.
\newblock {\em Degeneration of abelian varieties}, volume~22 of {\em Ergebnisse
  der Mathematik und ihrer Grenzgebiete (3) [Results in Mathematics and Related
  Areas (3)]}.
\newblock Springer-Verlag, Berlin, 1990.
\newblock With an appendix by David Mumford.
\newblock \href {http://dx.doi.org/10.1007/978-3-662-02632-8}
  {\path{doi:10.1007/978-3-662-02632-8}}.

\bibitem[FWG{\etalchar{+}}92]{faltings1984rational}
Gerd Faltings, Gisbert W\"ustholz, Fritz Grunewald, Norbert Schappacher, and
  Ulrich Stuhler.
\newblock {\em Rational points}.
\newblock Aspects of Mathematics, E6. Friedr. Vieweg \& Sohn, Braunschweig,
  third edition, 1992.
\newblock Papers from the seminar held at the Max-Planck-Institut f\"ur
  Mathematik, Bonn/Wuppertal, 1983/1984, With an appendix by W\"ustholz.
\newblock \href {http://dx.doi.org/10.1007/978-3-322-80340-5}
  {\path{doi:10.1007/978-3-322-80340-5}}.

\bibitem[GD64]{grothendieckegaIV_0}
Alexandre Grothendieck and Jean Dieudonn{\'e}.
\newblock {\'El\'ements de g\'eom\'etrie alg\'ebrique. IV. \'Etude locale des
  sch\'emas et des morphismes de sch\'emas. I}.
\newblock {\em Inst. Hautes \'Etudes Sci. Publ. Math.}, (20):259, 1964.
\newblock URL: \url{http://www.numdam.org/article/PMIHES_1964__20__5_0.pdf}.

\bibitem[GD66]{grothendieckegaIV_3}
Alexandre Grothendieck and Jean Dieudonn{\'e}.
\newblock {\'El\'ements de g\'eom\'etrie alg\'ebrique. IV. \'Etude locale des
  sch\'emas et des morphismes de sch\'emas. III}.
\newblock {\em Inst. Hautes \'Etudes Sci. Publ. Math.}, (28):255, 1966.
\newblock URL: \url{http://www.numdam.org/item?id=PMIHES_1966__28__255_0.pdf}.

\bibitem[Gro66]{grothendieck1966theoreme}
Alexandre Grothendieck.
\newblock Un th\'eor\`eme sur les homomorphismes de sch\'emas ab\'eliens.
\newblock {\em Invent. Math.}, 2:59--78, 1966.
\newblock \href {http://dx.doi.org/10.1007/BF01403390}
  {\path{doi:10.1007/BF01403390}}.

\bibitem[Gro74]{grothendieck1974groupes}
Alexandre Grothendieck.
\newblock {\em Groupes de {B}arsotti-{T}ate et cristaux de {D}ieudonn\'e}.
\newblock Les Presses de l'Universit\'e de Montr\'eal, Montreal, Que., 1974.
\newblock S\'eminaire de Math\'ematiques Sup\'erieures, No. 45 (\'Et\'e, 1970).
\newblock URL:
  \url{https://www.imj-prg.fr/~leila.schneps/grothendieckcircle/barsottitate.pdf}.

\bibitem[Gro05]{grothendieck2005cohomologie}
Alexander Grothendieck.
\newblock Cohomologie locale des faisceaux coh\'erents et th\'eor\`emes de
  {L}efschetz locaux et globaux ({SGA} 2).
\newblock 4:x+208, 2005.
\newblock S\'eminaire de G\'eom\'etrie Alg\'ebrique du Bois Marie, 1962,
  Augment\'e d'un expos\'e de Mich\`ele Raynaud. [With an expos\'e by Mich\`ele
  Raynaud], With a preface and edited by Yves Laszlo, Revised reprint of the
  1968 French original.
\newblock URL: \url{https://arxiv.org/abs/math/0511279}.

\bibitem[GRR06]{grothendieck2006groupes}
Alexandre Grothendieck, Michel Raynaud, and Dock~Sang Rim.
\newblock {\em {Groupes de Monodromie en Geometrie Algebrique: Seminaire de
  Geometrie Algebrique du Bois-Marie 1967-1969.(SGA 7)}}, volume 288.
\newblock Springer, 2006.

\bibitem[Har70]{hartshorneample}
Robin Hartshorne.
\newblock {\em Ample subvarieties of algebraic varieties}.
\newblock Lecture Notes in Mathematics, Vol. 156. Springer-Verlag, Berlin-New
  York, 1970.
\newblock Notes written in collaboration with C. Musili.

\bibitem[HM68]{hironakaformal}
Heisuke Hironaka and Hideyuki Matsumura.
\newblock Formal functions and formal embeddings.
\newblock {\em J. Math. Soc. Japan}, 20:52--82, 1968.
\newblock \href {http://dx.doi.org/10.2969/jmsj/02010052}
  {\path{doi:10.2969/jmsj/02010052}}.

\bibitem[HP18]{hartlpal}
Urs Hartl and Ambrus P\'al.
\newblock Crystalline {C}hebotarev density theorems.
\newblock {\em arXiv preprint arXiv:1811.07084v1}, 2018.

\bibitem[Kat79]{katz1979slope}
Nicholas~M. Katz.
\newblock Slope filtration of {$F$}-crystals.
\newblock In {\em Journ\'ees de {G}\'eom\'etrie {A}lg\'ebrique de {R}ennes
  ({R}ennes, 1978), {V}ol. {I}}, volume~63 of {\em Ast\'erisque}, pages
  113--163. Soc. Math. France, Paris, 1979.
\newblock Available at
  \url{https://web.math.princeton.edu/~nmk/old/f-crystals.pdf}.

\bibitem[Kat81]{katz1981serre}
N.~Katz.
\newblock Serre-{T}ate local moduli.
\newblock In {\em Algebraic surfaces ({O}rsay, 1976--78)}, volume 868 of {\em
  Lecture Notes in Math.}, pages 138--202. Springer, Berlin-New York, 1981.

\bibitem[Ked04]{kedlaya2001full}
Kiran~S. Kedlaya.
\newblock Full faithfulness for overconvergent {$F$}-isocrystals.
\newblock In {\em Geometric aspects of {D}work theory. {V}ol. {I}, {II}}, pages
  819--835. Walter de Gruyter, Berlin, 2004.

\bibitem[Ked16]{kedlaya2016notes}
Kiran~S Kedlaya.
\newblock Notes on isocrystals.
\newblock {\em arXiv preprint arXiv:1606.01321}, 2016.
\newblock [Online; accessed 11-December-2017].

\bibitem[Ked18]{kedlayacompanions}
Kiran~S Kedlaya.
\newblock {\'Etale and crystallline companions I}.
\newblock {\em draft}, 2018.
\newblock [Online; accessed 20-September-2018].
\newblock URL: \url{http://http://kskedlaya.org/papers/companions.pdf}.

\bibitem[KM74]{katzmessing}
Nicholas~M. {Katz} and William {Messing}.
\newblock {Some consequences of the Riemann hypothesis for varieties over
  finite fields.}
\newblock {\em {Invent. Math.}}, 23:73--77, 1974.

\bibitem[KM19]{kramermiller}
Joe Kramer-Miller.
\newblock Log decay {$F$}-isocrystals on higher dimensional varieties.
\newblock {\em arXiv preprint arXiv:1902:04730}, 2019.

\bibitem[Kos17]{koshikawa2015overconvergent}
Teruhisa Koshikawa.
\newblock Overconvergent unit-root {$F$}-isocrystals and isotriviality.
\newblock {\em Math. Res. Lett.}, 24(6):1707--1727, 2017.
\newblock \href {http://dx.doi.org/10.4310/MRL.2017.v24.n6.a7}
  {\path{doi:10.4310/MRL.2017.v24.n6.a7}}.

\bibitem[Kri17]{krishnamoorthy2017rank}
Raju Krishnamoorthy.
\newblock {Rank 2 local systems, Barsotti-Tate groups, and Shimura curves}.
\newblock {\em arXiv preprint arXiv:1711.04797}, 2017.

\bibitem[Laf02]{lafforgue2002chtoucas}
Laurent Lafforgue.
\newblock Chtoucas de {D}rinfeld et correspondance de {L}anglands.
\newblock {\em Invent. Math.}, 147(1):1--241, 2002.
\newblock \href {http://dx.doi.org/10.1007/s002220100174}
  {\path{doi:10.1007/s002220100174}}.

\bibitem[Laf11]{lafforgue2011estimees}
Vincent Lafforgue.
\newblock Estim\'ees pour les valuations {$p$}-adiques des valeurs propres des
  op\'erateurs de {H}ecke.
\newblock {\em Bull. Soc. Math. France}, 139(4):455--477, 2011.

\bibitem[LO98]{li1998moduli}
Ke-Zheng Li and Frans Oort.
\newblock {\em Moduli of supersingular abelian varieties}, volume 1680 of {\em
  Lecture Notes in Mathematics}.
\newblock Springer-Verlag, Berlin, 1998.
\newblock \href {http://dx.doi.org/10.1007/BFb0095931}
  {\path{doi:10.1007/BFb0095931}}.

\bibitem[MB81]{moret1981familles}
Laurent Moret-Bailly.
\newblock Familles de courbes et de vari{\'e}t{\'e}s ab{\'e}liennes sur
  {$\mathbb{P}^1$}.
\newblock {\em S{\'e}m. sur les pinceaux de courbes de genre au moins deux (ed.
  L. Szpiro). Ast{\'e}riques}, 86:109--140, 1981.

\bibitem[{Mor}85]{moret1985pinceaux}
Laurent {Moret-Bailly}.
\newblock {\em {Pinceaux de vari\'et\'es ab\'eliennes. (Pencils of abelian
  varieties).}}, volume 129.
\newblock Soci\'et\'e Math\'ematique de France (SMF), Paris, 1985.

\bibitem[NN81]{narasimhan1981polarisations}
M.~S. Narasimhan and M.~V. Nori.
\newblock Polarisations on an abelian variety.
\newblock volume~90, pages 125--128, 1981.
\newblock \href {http://dx.doi.org/10.1007/BF02837283}
  {\path{doi:10.1007/BF02837283}}.

\bibitem[Oor74]{oort1974}
Frans Oort.
\newblock Subvarieties of moduli spaces.
\newblock {\em Invent. Math.}, 24:95--119, 1974.
\newblock \href {http://dx.doi.org/10.1007/BF01404301}
  {\path{doi:10.1007/BF01404301}}.

\bibitem[P{\'a}l15]{pal2015monodromy}
Ambrus P{\'a}l.
\newblock The $ p $-adic monodromy group of abelian varieties over global
  function fields of characteristic $ p$.
\newblock {\em arXiv preprint arXiv:1512.03587}, 2015.

\bibitem[Shi67]{shimura1967}
Goro Shimura.
\newblock Algebraic number fields and symplectic discontinuous groups.
\newblock {\em Ann. of Math. (2)}, 86:503--592, 1967.
\newblock \href {http://dx.doi.org/10.2307/1970613}
  {\path{doi:10.2307/1970613}}.

\bibitem[Shi79]{shioda1979supersingular}
Tetsuji Shioda.
\newblock Supersingular {$K3$} surfaces.
\newblock In {\em Algebraic geometry ({P}roc. {S}ummer {M}eeting, {U}niv.
  {C}openhagen, {C}openhagen, 1978)}, volume 732 of {\em Lecture Notes in
  Math.}, pages 564--591. Springer, Berlin, 1979.

\bibitem[Sim91]{simpson1990nonabelian}
Carlos~T. Simpson.
\newblock Nonabelian {H}odge theory.
\newblock In {\em Proceedings of the {I}nternational {C}ongress of
  {M}athematicians, {V}ol. {I}, {II} ({K}yoto, 1990)}, pages 747--756. Math.
  Soc. Japan, Tokyo, 1991.

\bibitem[Sim92]{simpsonhiggs}
Carlos~T. Simpson.
\newblock Higgs bundles and local systems.
\newblock {\em Inst. Hautes \'{E}tudes Sci. Publ. Math.}, (75):5--95, 1992.
\newblock URL: \url{http://www.numdam.org/item?id=PMIHES_1992__75__5_0}.

\bibitem[Sos14]{sosna2014scalar}
Pawel Sosna.
\newblock Scalar extensions of triangulated categories.
\newblock {\em Applied categorical structures}, 22(1):211--227, 2014.

\bibitem[ST68]{serre1968good}
Jean-Pierre Serre and John Tate.
\newblock Good reduction of abelian varieties.
\newblock {\em Ann. of Math. (2)}, 88:492--517, 1968.
\newblock \href {http://dx.doi.org/10.2307/1970722}
  {\path{doi:10.2307/1970722}}.

\bibitem[ST18]{snowden2018constructing}
Andrew Snowden and Jacob Tsimerman.
\newblock Constructing elliptic curves from galois representations.
\newblock {\em Compositio Mathematica}, 154(10):2045--2054, 2018.
\newblock \href {http://dx.doi.org/10.1112/S0010437X18007315}
  {\path{doi:10.1112/S0010437X18007315}}.

\bibitem[{Sta}20]{stacks-project}
The {Stacks Project Authors}.
\newblock \textit{Stacks Project}.
\newblock \url{https://stacks.math.columbia.edu}, 2020.

\bibitem[Tat67]{tate1967p}
John Tate.
\newblock {$p$}-divisible groups.
\newblock In {\em Proc. {C}onf. {L}ocal {F}ields ({D}riebergen, 1966)}, pages
  158--183. Springer, Berlin, 1967.

\bibitem[Var02]{varshavskyshimura}
Yakov Varshavsky.
\newblock On the characterization of complex {S}himura varieties.
\newblock {\em Selecta Mathematica}, 8(2):283--314, 2002.

\end{thebibliography}

\end{document}